\documentclass[12pt, reqno]{amsart}

\usepackage{amssymb,amsmath,accents}
\usepackage{amsfonts,amsthm,mathrsfs}
\usepackage{cases}
\usepackage{graphicx} 
\usepackage{subcaption}
\usepackage{diagbox}
\usepackage{url}

\usepackage{txfonts}

\usepackage{xcolor}

\definecolor{dGREEN}{rgb}{0.0,0.6,0.4}

\numberwithin{equation}{section}
\newtheorem{thm}{Theorem}[section]
\newtheorem{cor}[thm]{Corollary}
\newtheorem{lemma}[thm]{Lemma}
\newtheorem{prop}[thm]{Proposition}

\newtheorem{remark}[thm]{Remark}

\DeclareMathOperator*{\osc}{osc}
\DeclareMathOperator*{\essinf}{ess.inf}
\DeclareMathOperator*{\esssup}{ess.sup}

\begin{document}
\title[Segmentation of monotone data by KWC]
{Segmentation of monotone data by Kobayashi-Warren-Carter type total variation energies}

\author[Y.~Giga]{Yoshikazu Giga}
\address[Y.~Giga]{Graduate School of Mathematical Sciences, The University of Tokyo, 3-8-1 Komaba, Meguro-ku, Tokyo 153-8914, Japan.}
\email{labgiga@ms.u-tokyo.ac.jp.}

\author[A.~Kubo]{Ayato Kubo}
\address[A.~Kubo]{Department of Mathematics, Faculty of Science, Hokkaido University, Kita 10, Nishi 8, Kita-Ku, Sapporo, Hokkaido, 060-0810, Japan.}
\email{kubo.ayato.j8@elms.hokudai.ac.jp}

\author[H.~Kuroda]{Hirotoshi Kuroda}
\address[H.~Kuroda]{Department of Mathematics, Faculty of Science, Hokkaido University, Kita 10, Nishi 8, Kita-Ku, Sapporo, Hokkaido, 060-0810, Japan.}
\email{kuro@math.sci.hokudai.ac.jp}

\author[K.~Sakakibara]{Koya Sakakibara}
\address[K.~Sakakibara]{Faculty of Mathematics and Physics, Institute of Science and Engineering, Kanazawa University, Kakuma-machi, Kanazawa-shi, Ishikawa 920-1192, Japan; 
RIKEN iTHEMS, 2-1 Hirosawa, Wako-shi, Saitama 351-0198, Japan.}
\email{ksakaki@se.kanazawa-u.ac.jp}


\maketitle
\thispagestyle{empty}

%
\begin{abstract}
We consider a Kobayashi-Warren-Carter (KWC) type total variation energy with a fidelity term.
 Since the energy is non-convex, the profiles of minimizers are quite different from those of the original Rudin-Osher-Fatemi energy.
 In one-dimensional setting, we prove that KWC type energy (and its generalization) with fidelity must have a piecewise constant minimizer if the data in fidelity is bounded not necessarily in $BV$.
 Moreover, we give quantitative estimates of the energy for a monotone data in fidelity.
 This estimate shows that any minimizer must be piecewise constant with an improved estimate of the number of jumps for a monotone data. 
 We also show the non-uniqueness of minimizers.
 Since this energy is useful from the point of segmentation or clustering, we compare with results of segmentation by the original Rudin-Osher-Fatemi energy and Mumford-Shah energy.
\end{abstract}

\maketitle


\section{Introduction} \label{SI} 

This paper is a continuation of our paper \cite{GKKOS} on minimizers of a Kobayashi-Warren-Carter (KWC) type total variation energy with a fidelity term.
 For a bounded domain $\Omega$ in $\mathbb{R}^n$ ($n\ge1$), we consider a (real-valued) integrable function $u$ with finite (essential) total variation in $\Omega$, i.e., its distributional derivative $Du$ is a finite Radon measure in $\Omega$.
 The space of all such functions is denoted by $BV(\Omega)$.
 Let $|Du|$ denote the total variation measure.
 For a given non-decreasing, lower semicontinuous function $K(\rho)$ for $\rho\ge0$, we consider a generalized total variation energy of the form
\[
	TV_K(u) = \int_{\Omega\backslash J_u} |Du|
	+ \int_{J_u} K \left( |u^+ - u^-| \right) d\mathcal{H}^{n-1},
\]
where $J_u$ denotes the (approximate) jump set of $u$ and $u^\pm$ is a trace of $u$ from each side of $J_u$; 
 $\mathcal{H}^{n-1}$ denotes the $n-1$ dimensional Hausdorff measure.
 For the precise meaning of $J_u$ and $u^\pm$, see \cite[Section 2]{AFP}.
 If $K(\rho)=\rho$, $TV_K$ agrees with standard total variation energy $TV$.
 For a given function $g\in L^2(\Omega)$ we consider a fidelity term 
\[
	\mathcal{F}(u)
	= \frac\lambda2 \int_\Omega |u - g|^2\, dx,
\]
where $\lambda>0$ is a constant.
 We consider
\[
	TV_{Kg}(u) = TV_K(u) + \mathcal{F}(u).
\]
In the case $TV_K=TV$, this energy (denoted by $TV_g$) is often called the Rudin-Osher-Fatemi functional which was proposed by \cite{ROF} to denoise the original image whose grey-level values equal $g$.
 As a singular limit of Kobayashi-Warren-Carter energy, we obtain
\[
	K(\rho) = \min_\xi \left( (\xi_+)^2 \rho + 2G(\xi) \right),
	\quad \rho \ge 0, \quad
	\xi_+ = \max(\xi, 0)
\]
with some non-negative function $G$ having strict minimum value $0$ at $\xi=1$;
 typically $G(\xi)=(\xi-1)^2/2$ \cite{GOU}, \cite{GOSU}.
 In this particular case
\begin{equation} \label{EK}
	K(\rho) = \rho / (1 + \rho).
\end{equation}
Note that $K(\rho)$ is no longer convex.

We are interested in profiles of a minimizer in the space $BV(\Omega)$.
 In the case of $K(\rho)=\rho$, $TV_g$ is strictly convex so regularity of a (unique) minimizer has been well studied.
 Since the paper \cite{GKKOS} includes several key references we do not repeat them.
 What is important is that $J_{u_*}\subset J_g$ for a minimizer $u_*$ of $TV_g$ for $g\in BV(\Omega)$; 
 see e.g.\ \cite{CL}.
 In particular, if $g$ is continuous, then $u_*$ is continuous in one-dimensional setting.

We consider a class of $K$ which includes \eqref{EK} as a particular example.
 Different from the case $K(\rho)=\rho$, a minimizer $U$ may not be continuous for continuous $g$.
 In fact, in \cite{GKKOS} it is shown in one-dimensional setting that $U$ is always piecewise constant if $g$ is continuous in the closure of $\Omega=(a,b)$.
 Moreover, the number of jumps of $U$ is estimated from above by a constant depending $\max g-\min g$ and $(b-a)\lambda$.
 Let us state this result of \cite{GKKOS} more precisely.
 As $K$ measuring a jump, we consider a function $K:[0,\infty)\to[0,\infty)$ satisfying
\begin{enumerate}
\item[(K1)] 
$K$ is lower semicontinuous, non-decreasing function with $K(0)=0$;
\item[(K2)] 
For any $M>0$, there exists a positive constant $C_M$ such that
\[
	K(\rho_1) + K(\rho_2) \ge K(\rho_1 + \rho_2) + C_M \rho_1 \rho_2
\]
for all $\rho_1,\rho_2\ge0$ with $\rho_1+\rho_2\le M$;
\item[(K3)] 
$\lim_{\rho\to0}K(\rho)/\rho=1$.
\end{enumerate}
The condition (K2) implies that $K$ is subadditive i.e.,
\[
	K(\rho_1) + K(\rho_2) \ge K(\rho_1 + \rho_2)
\]
but it is stronger than the subadditivity.
 Indeed, $K(\rho)=\rho$ is subadditive but it does not satisfy (K2).
 It is easy to see that our example \eqref{EK} satisfies (K2).
 The condition (K3) (with (K1)) implies that
\begin{enumerate}
\item[(K3w)] 
For any $M>0$, there exists $c_M>0$ such that
\[
	K(\rho) \ge c_M \rho
	\quad\text{for}\quad \rho \in [0, M].
\]
\end{enumerate} 
\begin{thm}[\cite{GKKOS}] \label{TEJ}
Assume that $K$ satisfies (K1), (K2) and (K3).
 Assume that $g\in C[a,b]$.
 Let $M>0$ be a number such that
\[
	M \ge \osc_{[a,b]} g
	:= \max_{[a,b]} g - \min_{[a,b]} g.
\]
Let $U\in BV(a,b)$ be a minimizer of $TV_{Kg}$.
 Then $U$ must be a piecewise constant function with finitely many jumps satisfying $\min g\le U\le\max g$ on $[a,b]$.
 Let $m$ be the number of jumps of $U$.
 Then
\[
	m \le \left[ (b-a) \lambda / A_M \right] + 1
\]
with $A_M:=\min\left\{ c_M / M, C_M \right\}$, where $C_M$ is in (K2) while $c_M$ is in (K3w).
 (The symbol $[r]$ denotes the integer part of $r\ge0$.)
 If $g$ is non-decreasing (resp.\ non-increasing), so is $U$.
 Moreover,
\[
	m \le \left[ (b-a) \lambda / C_M \right] + 1.
\]
\end{thm} 

This is a combination of the statements of \cite[Theorem 1.1 and Theorem 1.2]{GKKOS}.
 As noted in \cite{GKKOS} this result gives an upper bound of the number of segments in segmentation problem.
 As a clustering problem, this gives an upper bound of number of clusters. 
 If we weaken (K3) into (K3w) we still have the same estimate of the number $m$ of jumps if we restrict ourselves in the space of piecewise constant functions with possibly infinitely many jumps instead of whole $BV(\Omega)$.
 This is already noted in \cite[Theorem 1.4]{GKKOS} and it applies, for example, to Potts model when $K(\rho)\equiv1$ for $\rho>0$.
 It does not satisfy (K3) but it satisfies (K3w) as well as (K1) and (K2).

The purposes of this paper are as follows.
\begin{enumerate} 
\item[(i)] 
We give an explicit estimate for $TV_{Kg}$ when $g$ is a monotone function.
 This is done by approximation by piecewise constant functions.
 As a result we improve the estimate of $m$ by
\[
	m \le \left[ (b-a) \lambda / 2C_M \right] + 1.
\]
\item[(i\hspace{-0.1em}i)] 
As an application of our approximation, we prove that there exists a piecewise constant minimizer having estimates of $m$.
\item[(i\hspace{-0.1em}i\hspace{-0.1em}i)]
When $g(x)=x$, we prove that all jumps $u^+-u^-$ are the same.
 Moreover, we show that minimizers are not unique.
\item[(i\hspace{-0.1em}v)]
We give a numerical experiment to compare segmentation results by $TV_g$, $TV_{Kg}$ and Mumford-Shah functional.
\end{enumerate} 

Let us state results (i), (i\hspace{-0.1em}i), (i\hspace{-0.1em}i\hspace{-0.1em}i) in a precise way.
 In the sequel, we often use the notation $w(a\pm0)$ defined by
\[
    w(a \pm 0) := \lim_{\delta\downarrow0} w(a\pm\delta),
\]
where $w$ is a function of one real variable.
 If $g$ is a monotone function, we are able to estimate $TV_{Kg}$.
\begin{lemma} \label{LWEBJ}
Assume that $K$ satisfies (K1), (K2) and (K3).
 Assume that $g\in C[a,b]$ is non-decreasing.
 Let $M>0$ be taken so that $M\ge g(b)-g(a)$.
 Assume that $C_*=C_M-\lambda(b-a)/2>0$.
 Let $v$ be a non-decreasing function on $(a,b)$ satisfying $v(\alpha+0)=g(\alpha)$, $v(\beta-0)=g(\beta)$ for some $\alpha<\beta$ with $\alpha\ge a$, $\beta\le b$.
\[
	TV_{Kg}(v) \ge C_* f(v)
	+ K \left( \left| v(\beta-0) - v(\alpha+0) \right| \right)
\]
with some nonnegative explicit function $f$ depending on jumps of $v$.
 Moreover $f(v)=0$ if and only if $v$ has only one jump in $(\alpha,\beta)$.
\end{lemma} 
As an application, we have an improved estimate for $m$.

\begin{thm} \label{TMainMon}
Assume that $K$ satisfies (K1), (K2) and (K3).
 Assume that $g\in C[a,b]$ is non-decreasing (resp.\ non-increasing).
 Let $M$ be taken such that $	M\geq\operatorname{osc}_{[a,b]}g=\left|g(b)-g(a)\right|$.
 Let $U\in BV(a,b)$ be a minimizer of $TV_{Kg}$.
 Then $U$ must be a non-decreasing (resp.\ non-increasing) piecewise constant function satisfying $\inf g\leq U\leq\sup g$ on $[a,b]$.
 The number $m$ of jumps is estimated as
\[
	m \leq \left[(b-a)\lambda/(2C_M) \right] + 1.
\]
\end{thm}
This can be proved by Lemma~\ref{LWEBJ} and approximation.
 If $g$ is merely bounded, i.e., $g\in L^\infty(a,b)$, then by approximation and Lemma~\ref{LWEBJ}, we are able to conclude that Theorem~\ref{TMainMon} is still valid for $g\in L^\infty(a,b)$.
\begin{cor} \label{CMain}
The conclusion of Theorem~\ref{TMainMon} still valid for $g\in L^\infty(a,b)$ instead of $g\in C[a,b]$.
\end{cor}

For non-monotone discontinuous $g$, we do not have quantitative estimate so we have a weaker result which does not exclude the possibility of existence of non-piecewise constant minimizer.
\begin{cor} \label{CMain2}
Assume that $K$ satisfies (K1), (K2) and (K3).
 Assume that $g\in L^\infty(a,b)$.
 Then there exists a piecewise constant minimizer $U$ with finitely many jumps for $TV_{Kg}$.
 Let $M>0$ be a number such that
\[
	M \ge \osc_{[a,b]} g
	= \esssup_{[a,b]} g - \essinf_{[a,b]} g.
\]
Let $U$ be a piecewise constant minimizer of finitely many jumps.
 Then it must satisfy
\[
	\essinf g \le U \le \esssup g
\]
on $[a,b]$.
 Moreover,
\[
	m \le \left[(b-a)\lambda/A_M \right] + 1
\]
for the number of jumps of $U$, where $A_M$ is defined in Theorem~\ref{TEJ}.
\end{cor}

We next observe that all jumps $u^+-u^-$ has the same size if $g(x)$ is linear.
 Moreover, by an explicit calculation, we found an example that minimizers are not unique.
 We state it just for the case $K(\rho)=\rho/(1+\rho)$. 
\begin{thm} \label{TSJ}
Assume that $g$ is linear with $g'>0$ and $K(\rho)=\rho/(1+\rho)$.
 For $\Omega=(a,b)$, let $U$ be a non-decreasing piecewise constant minimizer of $TV_{Kg}$ which only jumps at $a<a_1<a_2<\cdots<a_m<b$.
 Then $a_{i+1}-a_i$ ($i=1,\ldots,m-1$) and jump size $U(a_i+0)-U(a_i-0)$ ($i=1,\ldots,m$) are independent of $i$.
\end{thm}

As an application let us state a non-uniqueness result.
\begin{cor} \label{CSJ}
Under the same hypothesis of Theorem~\ref{TSJ}, for some fidelity parameter $\lambda$, there exist two minimizers of $TV_{Kg}$ for $g(x)=x$.
 One has only one jump and the other has two jumps.
\end{cor}

It seems that the segmentation by $TV_{Kg}$ is much clearer compared with segmentation by Mumford-Shah functional or $TV_g$.
 However, this needs to be studied further especially multi-dimensional case.

To demonstrate the practical implications of our theoretical findings and highlight the unique properties of the KWC type energy, we perform numerical experiments in Section \ref{SNE}. 
Since directly computing the $L^2$-gradient flow for non-convex energies with jump sets is numerically challenging, we utilize its phase-field approximation and compare the results with the standard ROF model and the Ambrosio-Tortorelli approximation of the Mumford-Shah functional.
First, we numerically verify Theorem \ref{TSJ} and Corollary \ref{CSJ} for linear data, confirming the perfectly uniform jump sizes and the non-uniqueness of the global minimizers. 
Furthermore, we apply these models to non-monotone oscillating data and noisy piecewise constant signals. The results clearly illustrate that the KWC model robustly segments complex data into perfectly flat piecewise constant blocks without suffering from the severe staircasing effect inherent in the ROF model or the edge-blurring effect of the Mumford-Shah model. 
This visually and strongly confirms the powerful clustering capability of KWC-type energies as guaranteed by our theoretical results.

This paper is organized as follows.
 In Section~\ref{SA}, we discuss approximation of $TV_K$ by piecewise constant function.
 In Section~\ref{SB}, we give several basic estimates.
 In Section~\ref{SQ}, we prove a precise form (Lemma~\ref{LEBJ}) of Lemma~\ref{LWEBJ}.
 The proofs of Theorem~\ref{TMainMon}, Corollary~\ref{CMain}, Corollary~\ref{CMain2} and Theorem~\ref{TSJ} are given in this section.
 In the last section, we give a few numerical experiments to verify non-uniqueness of minimizers.
 We also compare results of segmentation by Mumford-Shah energy, $TV_g$ and $TV_{Kg}$.
%
\section{Approximation} \label{SA} 

We begin with a simple approximation for a general function by a piecewise constant function such that $TV_K$ is also approximated.
\begin{lemma} \label{LAp1}
Assume that $K$ satisfies (K1), (K3) and that $\Omega=(a,b)$.
 For any $u\in(L^p\cap BV)(\Omega)$ with $p\geq1$, there exists a sequence of piecewise constant functions $\{u_m\}$ (with finitely many jumps) such that $u_m\to u$ in $L^p(\Omega)$ and $TV_K(u_m)\to TV_K(u)$ as $m\to\infty$.
\end{lemma}
Since $u_m\to u$ in $L^2(\Omega)$ implies $\mathcal{F}(u_m)\to\mathcal{F}(u)$ when $g\in L^2(\Omega)$, this lemma yields
\begin{lemma} \label{LAp2}
Assume that $K$ satisfies (K1), (K3) and $\Omega=(a,b)$.
 For any $u\in(L^2\cap BV)(\Omega)$, there exists a sequence of piecewise constant functions $\{u_m\}$ such that $u_m\to u$ in $L^2(\Omega)$ and $TV_{Kg}(u_m)\to TV_{Kg}(u)$ as $m\to\infty$.
\end{lemma}
We shall prove Lemma~\ref{LAp1} by reducing the problem when $u$ is continuous, i.e., $u\in C[a,b]$.
\begin{lemma} \label{LApC}
Assume that $K$ satisfies (K3) and $\Omega=(a,b)$.
 For any $u\in C[a,b]\cap BV(\Omega)$, there exists a sequence of piecewise constant functions $\{u_m\}$ with $u_m(a+0)=u(a)$, $u_m(b-0)=u(b)$ such that $u_m\to u$ in $C[a,b]$ and $TV_K(u_m)\to TV_K(u)$ as $m\to\infty$.
\end{lemma}
\begin{proof}
We first construct an approximate sequence $\{u_m\}$ without requiring that $u_m(a)=u(a)$, $u_m(b)=b$.
For a continuous function $u$, $TV_K(u)$ agrees with usual $TV(u)$.
 We extend $u$ continuously in some neighborhood of $\overline{\Omega}$ and denote it by $\overline{u}$.
 We mollify $\overline{u}$ by a symmetric mollifier $\rho_\varepsilon$.
 It is well known that $u_\varepsilon=\overline{u}*\rho_\varepsilon$ is $C^\infty$ in $[a,b]$ and $u_\varepsilon\to u$ in $C[a,b]$ as $\varepsilon\to0$.
 Moreover, $TV(u_\varepsilon)\to TV(u)$ \cite[Proposition 1.15]{Giu}.
 Since $\rho_\varepsilon$ can be approximated (in $C^1$ sense) by polynomials in a bounded set, we approximate $u_\varepsilon$ by a polynomial with its derivative in $C[a,b]$.
 Thus, we may assume that $u$ is a polynomial.

We divide the interval $(a,b)$ into finitely many subintervals $\left\{(a_i,a_{i+1})\right\}_{i=0}^\ell$ with $a=a_0<a_1< \cdots <a_\ell<a_{\ell+1}=b$ such that on each such an interval $u$ is either increasing or decreasing.
 This is possible since $u$ is a polynomial.
 For a given $\eta>0$, we define a function
\[
	u^\eta(x) = k\eta \quad\text{if}\quad
	k\eta \leq u(x) < (k+1)\eta
	\quad\text{for}\quad k \in \mathbb{Z}.
\]
There might be a chance that the set where $u^\eta=k\eta$ contains an isolated point for some $k$ if $u$ takes local maximum value $k\eta$.
 Since we consider essential total variation, we may assume that the set where $u^\eta=k\eta$ is a non-trivial interval and $u^\eta$ is piecewise constant by considering its lower semicontinuous envelope i.e., the greatest lower semicontinuous function less than $u^\eta$.
 This function $u^\eta$ is piecewise constant.
 Let $TV\left(u,(a_i,a_{i+1})\right)$ denote the total variation of $u$ in $(a_i,a_{i+1})$ for $i=0,\ldots,\ell$.
 Then
\[
	\left| TV\left(u,(a_i,a_{i+1})\right) - TV\left(u^\eta,(a_i,a_{i+1})\right) \right| \leq \eta.
\]
By the assumption (K3), we see that for any $\delta>0$ there is $\eta_0>0$ such that $\left|\eta-K(\eta)\right|<\delta\eta$ for $\eta<\eta_0$.
 Since $u$ is continuous, the size of jumps of $u^\eta$ is always $\eta$ so
\[
	\left\lvert TV \left(u^\eta,(a_i,a_{i+1})\right) - TV_K \left(u^\eta,(a_i,a_{i+1})\right) \right\rvert
	\leq \delta TV \left(u^\eta, (a_i,a_{i+1})\right)
	\quad\text{for}\quad \eta < \eta_0,
\]
where we use the same convention to $TV_K$.
 We thus observe that
\begin{gather*}
	\left\lvert TV (u) - TV_K (u^\eta) \right\rvert
	\leq \sum_{i=0}^\ell \left(\eta + \delta TV \left(u^\eta,(a_i,a_{i+1})\right) \right) \\
	= \eta (\ell+1) + \delta TV(u^\eta)
	\leq \eta (\ell+1) + \delta \left(TV(u)+\eta\right).
\end{gather*}
Sending $\eta\to0$, we now conclude that
\[
	\varlimsup_{\eta\to0} \left\lvert TV(u) - TV_K(u^\eta) \right\rvert
	\leq \delta TV(u).
\]
Since $\delta>0$ is arbitrary, the convergence $TV_K(u^\eta)\to TV(u)$ as $\eta\to0$ follows.
 By definition, $u^\eta\to u$ in $C[a,b]$.
 We now obtain a sequence of piecewise constant functions $\{u_m\}$ such that $u_m\to u$ in $C[a,b]$ and $TV_K(u_m)\to TV_K(u)$ as $m\to\infty$.

It remains to modify $u_m$ such that $u_m(a)=u(a)$ and $u_m(b)=u(b)$.
 By adding a constant, we may assume $u(a)=0$.
 We take a monotone piecewise constant approximation of a linear function
\[
	r_m(x) = \frac{u_m(b) -u(b)-u_m(a)}{b-a} (x-a)
	+ u_m(a)
\]
whose jump points agree with those of $u_m$.
 Its total variation tends to zero as $m\to\infty$ since $u_m\to u$ in $C[a,b]$.
 Comparing with $TV$ and $TV_K$ as above, we conclude that $\bar{u}_m=u_m-r_m(x)$ is the desired sequence satisfying $\bar{u}_m(a)=u(a)$, $\bar{u}_m(b)=u(b)$.
 The proof is now complete. 
\end{proof}
\begin{proof}[Proof of Lemma~\ref{LAp1}]

Since $u$ is bounded by $u\in BV(\Omega)$, for any $\delta>0$, the set $J_\delta$ of jump discontinuities of $u$ whose jump size greater than $\delta$ is a finite set.
 For any $\varepsilon>0$, we take $\delta>0$ such that
\[
	\sum_{x\in J_u\backslash J_\delta} K\left(|u^+-u^-|(x)\right) < \varepsilon.
\]
This is possible since $K$ is continuous at $0$ and $K(0)=0$ by (K1).
 We may assume that $J_\delta=\{a_j\}_{j=1}^\ell$ with $a_j<a_{j+1}$ and $a_0=a$, $a_{\ell+1}=b$.
 In each interval $(a_j,a_{j+1})$, we approximate $u$ in $L^p$ with continuous function $u_\varepsilon$ such that
 \[
    TV\left( u_\varepsilon,(a_j,a_{j+1}) \right)
    \to TV\left( u,(a_j,a_{j+1}) \right)
    \quad\text{as}\quad \varepsilon \to 0
 \]
keeping $u_\varepsilon(a_j)=u(a_j+0)$, $u_\varepsilon(a_{j+1})=u(a_{j+1}-0)$.
 We now apply Lemma~\ref{LApC} on each interval $(a_j,a_{j+1})$ to approximate $u_\varepsilon$ by a piecewise constant function $(u_\varepsilon)_m$.
 The jump at $a_j$ of $(u_\varepsilon)_m$ is exactly equal to
\[
	|u^+ - u^-| (a_j)
\]
since $(u_\varepsilon)_m^\pm (a_j)= u^\pm(a_j)$.
 We now obtain a desired sequence of piecewise constant functions.
\end{proof}
If $u$ is non-decreasing, it is rather clear that $u_m$ in Lemmas~\ref{LAp1}, \ref{LAp2}, \ref{LApC} can be taken as a non-decreasing function by construction.
\begin{cor} \label{CApMon}
Assume that $K$ satisfies (K1), (K3) and that $\Omega=(a,b)$.
 Assume that $u$ is non-decreasing (resp.\ non-increasing) and bounded in $(a,b)$.
 Then there is a sequence of non-decreasing (non-increasing) piecewise constant functions $\{u_m\}$ (with finitely many jumps) such that $u_m\to u$ in $L^p(\Omega)$ for any $p\ge1$ and $TV_K(u_m)\to TV_K(u)$ as $m\to\infty$.
\end{cor}
\begin{remark}[truncation] \label{RTru}
In Lemma~\ref{LAp1} (and also in Lemma~\ref{LAp2}, Lemma~\ref{LApC} and Corollary~\ref{CApMon}), if we further assume that $K$ is subadditive, then $u_m$ can be taken so that
\[
	q_- \le u_m \le q_+
	\quad\text{with}\quad q_- = \essinf_\Omega u,\quad 
	q_+ = \esssup_\Omega u,
\]
where the above first estimate holds on $\Omega$ almost everywhere.
 Indeed, we truncate $u_m$ as
\[
	\bar{\bar{u}}_m = \min(\bar{u}_m,q_+),\quad
	\bar{u}_m = \max(u_m, q_-).
\]
Since $|\bar{\bar{u}}_m-u|\le|u_m-u|$ almost everywhere on $\Omega$, the convergence $
\bar{\bar{u}}_m\to u$ in $L^p(\Omega)$ is clear.
 It remains to prove $TV_K(\bar{\bar{u}}_m)\to TV_K(u)$.
 By monotonicity of $K$ in (K1), we observe that
\[
	TV_K(\bar{\bar{u}}_m) \le TV_K(u_m).
\]
Since $TV_K(u_m)$ is bounded and $\bar{\bar{u}}_m$ is bounded, we may assume that $\bar{\bar{u}}_m$ weakly$^*$ converges to $u$ in $BV(\Omega)$ \cite[Proposition 2.3]{GKKOS}.
 By lower semicontinuity \cite[Proposition 2.4]{GKKOS}, we see that
\[
	TV_K(u) \le \varliminf_{m\to\infty} TV_K(\bar{\bar{u}}_m);
\]
here we invoke the assumption of subadditibity of $K$.
Since $TV_K(\bar{\bar{u}}_m)\le TV_K(u_m)$, this now implies that $TV_K(\bar{\bar{u}}_m)\to TV_K(u)$ as $m\to\infty$.
\end{remark}
\begin{remark}[at the boundary] \label{RApMon}
In Corollary~\ref{CApMon}, if we further assume that $K$ is subadditive, and $u$ is continuous at $a$ and $b$, the function $u_m$ can be taken so that $u_m(a+0)=u(a)$, $u_m(b-0)=u(b)$.
 To show this statement, we may assume that $u(a)<u(b)$.
 By Remark~\ref{RTru}, we may assume that $u(a)\le u_m\le u(b)$ on $\Omega$.
 We shall modify $u_m$ near $x=a$ and $b$ so that $u_m(a+0)=u_m(a)$, $u_m(b-0)=u(b)$.
 Since the argument is symmetric, we only discuss modification near $x=a$.
 We may assume that $u(a)=0$ by adding a constant.
 Since $u_m\to u$ in $L^p(\Omega)$, we may assume that $u_m\to u$ in $\Omega$ almost everywhere, by taking a subsequence.
 Let $a_m$ be the first jump of $u_m$; i.e., $u_m$ is a constant on $(a,a_m)$ and $u_m(a_m+0)>u_m(a_m-0)$($=u_m(a+0)$).

Since $u$ is continuous at $a$ with $u(a)=0$, $u_m(a+0)\to0$ as $m\to\infty$.
 Indeed, if not there would exist $\delta>0$ such that $u_m(a+0)>\delta$.
 By monotonicity of $u_m$ this would imply that  $\lim_{m\to\infty}u_m:=u\ge\delta$ almost everywhere.
 This contradicts the continuity of $u$ at $a$ with $u(a)=0$.

We modify $u_m$ as
\[
	\bar{u}_m(x) =\left\{
\begin{array}{ll}
	u(a)\text{(}=0\text{)}, & 0 \le x < \min(a_m/2,1/m) =: \alpha_m \\
	u_m(x), & \alpha_m \le x < b.
\end{array}
\right.
\]
Since $\alpha_m\to0$, $\|\bar{u}_m-u_m \|_{L^p}\to0$ as $m\to\infty$.
 Since $u_m(a+0)\to0$,
\[
	TV_K(\bar{u}_m) - TV_K(u_m) = K\left(u_m(a+0)\right)\to0
\]
as $m\to\infty$.
 Thus $TV_K(\bar{u}_m)\to TV_K(u)$ as well as $\bar{u}_m\to u$ in $L^p(\Omega)$.
\end{remark}

We next observe that the way of convergence of jumps becomes very simple if $K$ fulfills (K2).
 For a bounded non-decreasing function $u$ in $\Omega$, we take its left continuous representative.
 Let $J_u$ denote the set of jump discontinuities of $u$ which is either a finite set or a countable set.
 It is of the form $J_u=\{z_j\}_{j=1}^\infty$.
 Let $I_\lambda(\rho)$ denote an interval $(\lambda,\lambda+\rho]$.
 If we set $\lambda_j=u(z_j)=u(z_j-0)$, then
\[
    \left(u(z_j), u(z_j+0) \right]=I_{\lambda_j}(\rho_j)
    \quad\text{with}\quad
    \rho_j = u(z_j+0) - u(z_j-0).
\]
We call $I_{\lambda_j}(\rho_j)$ a jump interval of $u$ (at $z_j$).
 It is convenient to consider ``inverse function" of $u$.
 We define an inverse function by
\[
    u^{-1}(p) = \inf \left\{ x \in \Omega \mid
    u(x)> p \right\}
\]
which is left continuous.
 By definition, $u^{-1}$ is constant on a jump interval $I_{\lambda_j}(\rho_j)$.
 Let $u_m$ be a non-decreasing piecewise constant function on $(a,b)$, which is assumed to be left continuous.
 Let $\{z_i^m\}_{i=1}^{n_m}$ denote its jump discontinuities.
 We may assume that $z_i^m<z_{i+1}^m$ for $i=1,\ldots,n_m-1$.
 We set
\[
    \lambda_i^m = u_m(z_i^m), \quad
    \rho_i^m = u_m(z_i^m+0) - u_m(z_i^m)
\]
so that
\[
    \left(u_m(z_i^m), u_m(z_i^m+0) \right]
    = I_{\lambda_i^m} (\rho_i^m).
\]
Since $u_m$ is piecewise constant, we observe that
\[
    \sum_{i=1}^{n_m} \rho_i^m
    = u_m(b-0) - u_m(a+0).
\]
\begin{thm} \label{TCJ}
Assume the same hypotheses of Corollary~\ref{CApMon} concerning $K$, $\Omega$, $u$ and $u_m$.
 Assume further (K2).
 Then for each jump interval $I_{\lambda_j}(\rho_j)$, there exists a subsequence $\{i_m^j\}_{m=1}^\infty\subset\mathbb{N}$ such that $\left\{I_{\lambda_{i_m^j}^m}(\rho_{i_m^j}^m)\right\}_{m=1}^\infty$ converges to $I_{\lambda_j}(\rho_j)$ as $m\to\infty$.
 In other words,
\[
    \lambda_{i_m^j}^m \to \lambda_j, \quad
    \rho_{i_m^j}^m \to \rho_j
    \quad\text{as}\quad m \to \infty.
\]
Moreover, if a subsequence $\{i_m\}$ is taken such that
\[
    \lambda_{i_m}^m \to \lambda, \quad
    \rho_{i_m}^m \to \rho
    \quad\text{as}\quad m \to \infty
\]
with some $\lambda$ and $\rho$ and $\lambda\neq\lambda_j$ for any $j$, then $\rho=0$.
 The convergence $\rho_{i_m}^m\to0$ is uniform, i.e., 
\[
    \lim_{m\to\infty} \max \left\{\; \rho_\ell^m \bigm|
    \ell\neq i_m^j\ \text{for all}\ j\; \right\} = 0.
\] 
\end{thm}
\begin{proof}
Since $u_m\to u$ in $L^1$, $u_m^{-1}\to u^{-1}$ in $L^1(G)$ as $m\to\infty$ for any closed interval $G \subset\left(u(a+0),u(b-0)\right)$.
 By definition, $u^{-1}=z_j$ in $I_{\lambda_j}(\rho_j)$ and $u^{-1}(p)<z_j$ for $p<\lambda_j$ and $u^{-1}(p)>z_j$ for $p>\lambda_j+\rho_j$.
 The $L^1$-convergence implies that there is a subsequence $\{i_m^j\}$, $\{\bar{i}_m^j\}$ with $i_m^j\le \bar{i}_m^j$ such that
\[
    p_m = u_m(z_{i_m^j}^m) \to \lambda_j \quad
    q_m := u_m(z_{\bar{i}_m^j}^m + 0) \to \lambda_j + \rho_j
\]
as $m\to\infty$.

We consider $TV_K$ near $z_j$.
 For $\varepsilon>0$, we take $\delta>0$ small such that
\begin{equation} \label{EATVK}
    TV_K(u,Z_j^\delta) \le K(\rho_j) + \varepsilon, \quad
    Z_j^\delta = (z_j-\delta, z_j+\delta).
\end{equation}
By definition, on the interval
\[
    (p'_m, q'_m]
    = \left( u_m(z_{i'_m}^m), u_m(z_{i'_m}^m + 0) \right],
\]
$u_m^{-1}$ equals a constant $z_{i'_m}^m$.
 Assume that $i_m^j\le i'_m\le \bar{i}_m^j$.
 Since $u_m^{-1}$ converges to $u^{-1}$ in $L^1$,
 we may assume that $p'_m$ and $q'_m$ converge to some $p$, $q$ as $m\to\infty$ with
\[
    \lambda_j \le p, \quad
    q \le \lambda_j + \rho_j;
\]
see Figure~\ref{FAp}.
\begin{figure}[tb]
\centering 
\includegraphics[keepaspectratio, scale=0.3]{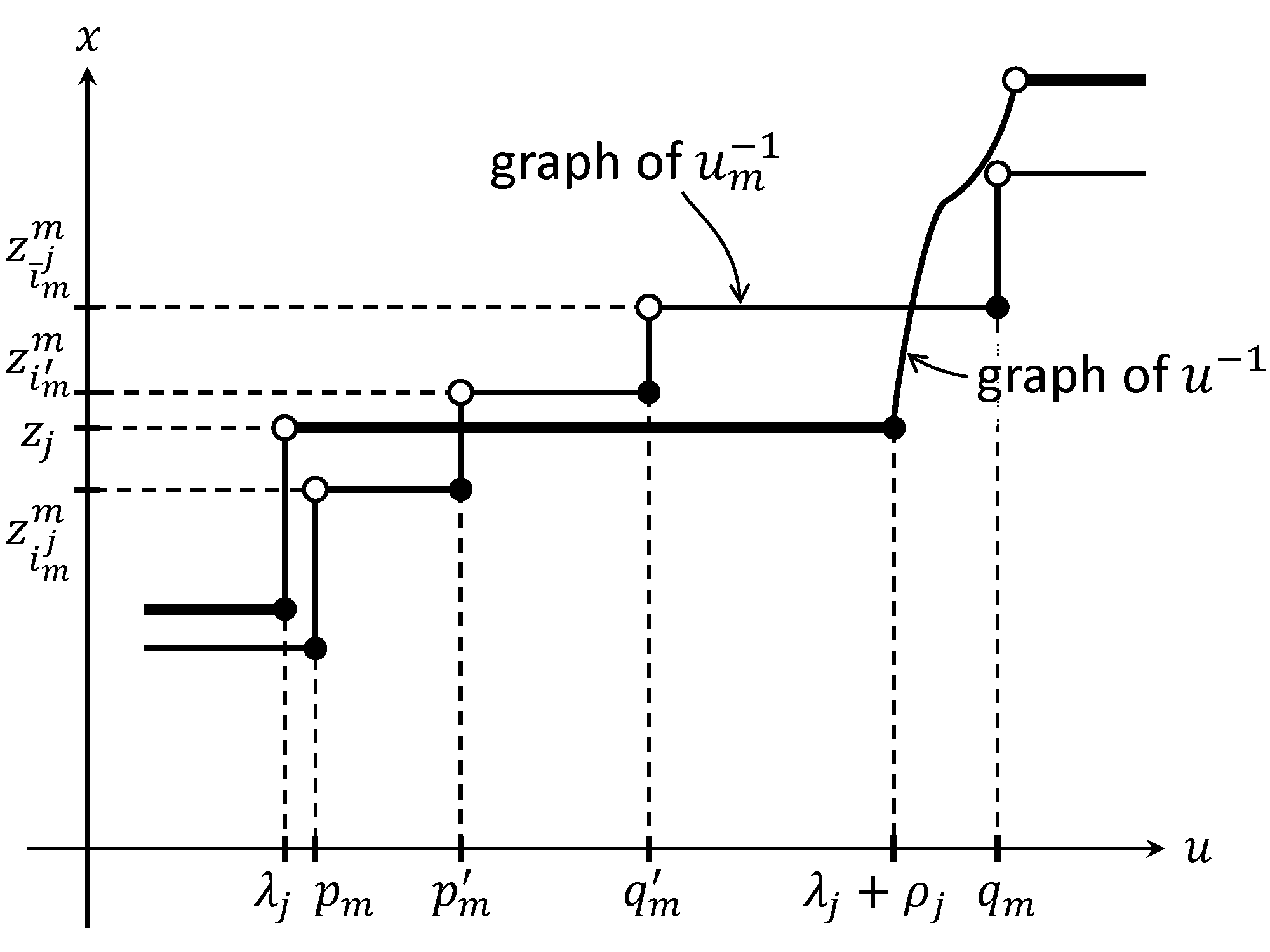}
\caption{the graph of $u^{-1}$ and $u_m^{-1}$\label{FAp}}
\end{figure}

\noindent
\emph{Case} 1 ($p<q$).
 We shall prove that $\lambda_j=p$, $\rho_j=q-p$.
 We set $q_m-p_m=:\rho^m$, $q'_m-p'_m=:\rho'_m$.
 By definition of $Z_j^\delta$, for sufficiently large $m$ we see that
\[
    TV_K(u_m,Z_j^\delta) \ge K(\rho'_m)
    + K(\rho^m -\rho'_m).
\]
Since $TV_K(u_m)\to TV_K(u)$ as $m\to\infty$ implies
\[
    \lim_{m\to\infty} TV_K (u_m,Z_j^\delta)
    = TV_K (u,Z_j^\delta),
\]
by (K1), sending $m\to\infty$ yields
\[
    TV_K(u,Z_j^\delta)
    = \varliminf_{m\to\infty} TV_K(u_m,Z_j^\delta)
    \ge K((\rho_*) + K(\rho_j-\rho_*)
\]
where $\rho_*=\lim_{m\to\infty}\rho'_m$.
 By our assumption $\rho_*=q-p>0$ and (K2) we see that
\[
    K(\rho_*) + K(\rho_j-\rho_*)
    \ge K(\rho_j) + C_M \rho_*(\rho_j-\rho_*)
\]
where $M=\rho_j$.
 Thus,
\[
    TV_K(u,Z_j^\delta) \ge K(\rho_j)
    + C_M \rho_*(\rho_j-\rho_*).
\]
By \eqref{EATVK}, we conclude that $\rho_j-\rho_*=0$.
 Thus $\lambda_j=p$, $\rho_j=q-p$.

\noindent
\emph{Case} 2 ($p=q$).
 Assume that for all subsequences $i'_m$, $q'_m-p'_m\to0$.
 We estimate
\[
    TV_K(u_m, Z_j^\delta)
    = \sum_{\ell=0}^{k_m} K(\rho_\ell^m)
\]
with $\rho_\ell^m=q'_{m_\ell}-p'_{m_\ell}$,
\[
    p'_{m_\ell} = u_m(z_{i_m^j+\ell}^m), \quad
    q'_{m_\ell} = u_m(z_{i_m^j+\ell}^m+0), \quad
    i_m + k_m = \bar{i}_m^j, \quad
    \ell=0,\ldots,k_m.
\]
By our assumption of case 2, $\rho_\ell^m\to0$ as $m\to\infty$.
 Moreover, the convergence is uniform, i.e.,
\begin{equation} \label{EUni}
    \lim_{m\to\infty} \max_{0\le\ell\le k_m} \rho_\ell^m = 0.
\end{equation}
Indeed, if otherwise, there were an interval $(p'_{m_\ell},q'_{m_\ell})$ which subsequently converges to an interval with positive length as $m\to\infty$.
 By $L^1$ convergence of $u_m^{-1}$, this convergence is a full convergence which is reduced to Case 1.
 This would contradict $p=q$ so \eqref{EUni} is proved.

By induction (see Proposition~\ref{PInd} below), we observe that
\[
    \sum_{\ell=0}^{k_m} K(\rho_\ell^m)
    \ge K \left(\sum_{\ell=0}^{k_m} \rho_\ell^m \right)
    + C_M \sum_{0\le i<\ell\le k_m} \rho_i^m \rho_\ell^m, \quad
    M \ge 2\sum_{\ell=0}^{k_m} \rho_\ell^m.
\]
Since
\[
    \sum_{\ell=0}^{k_m} \rho_\ell^m
    = q_m - p_m =: s_m
\]
and $s_m\to\rho_j$, we see that
\[
    \sum_{0\le i<\ell\le k_m} \rho_i^m \rho_\ell^m
    = \frac12 \left( s_m^2 - \sum_{\ell=0}^{k_m} (\rho_\ell^m)^2 \right) \to \frac12 \rho_j^2
    \quad\text{as}\quad m \to \infty
\]
since
\[
    \sum_{\ell=0}^{k_m} (\rho_\ell^m)^2
    \le \max_{1\le\ell\le k_m} \rho_\ell^m \cdot s_m \to 0
\]
by \eqref{EUni}.
 Thus
\[
    TV_K(u,Z_j^\delta) \ge K(\rho_j) + \frac{C_M \rho_j^2}{2}.
\]
Again, this contradicts \eqref{EATVK} by taking $\varepsilon$ smaller than $C_M\rho_j^2/2$.
 We thus conclude that the Case 2 does not occur.

We now conclude the existence of a sequence $\left\{I_{\lambda_{i_m^j}^m}(\rho_{i_m^j}^m)\right\}$ converging to $I_{\lambda_j}(\rho_j)$ as $m\to\infty$.

It remains to prove the last statement.
 If such a sequence exists, by $L^1$-convergence of $u_m^{-1}$ the limit must be one of $I_{\lambda_j}(\rho_j)$ if the size does not converge to zero.
 By \eqref{EUni}, the proof of Theorem~\ref{TCJ} is now complete.
\end{proof}
\begin{prop} \label{PInd}
Assume that $K$ satisfies (K2).
 Let $M>0$.
 If $\rho_i\ge0$ ($0\le i \le k$) satisfies $\sum_{i=0}^k\rho_i\le M$, then
\begin{equation} \label{EInd}
    \sum_{i=0}^k K(\rho_i)
    \ge K \left(\sum_{i=0}^k \rho_i \right)
    + C_M \sum_{0\le i<\ell\le k} \rho_i \rho_\ell
\end{equation}
for $k\ge1$.
\end{prop}
\begin{proof}
The case $k=1$ is nothing but (K2).
 Assume that \eqref{EInd} holds for a fixed $k$ ($\ge1$).
 Then
\begin{align*}
    \sum_{i=0}^{k+1} K(\rho_i)
    &\ge K\left(\sum_{i=0}^k \rho_i \right) + K(\rho_{k+1})
    + C_M \sum_{0\le i<\ell \le k} \rho_i \rho_\ell
    \quad\text{(by induction)} \\
    &\ge K\left(\sum_{i=0}^{k+1} \rho_i\right)
    + C_M \left(\sum_{i=0}^k \rho_i\right) \rho_{k+1}
    + C_M \sum_{0\le i<\ell\le k} \rho_i \rho_\ell
    \quad\text{(by (K2))} \\
    &= K\left(\sum_{i=0}^{k+1} \rho_i \right)
    + C_M \sum_{0\le i<\ell\le k+1} \rho_i \rho_\ell.
\end{align*}
The proof is now complete.
\end{proof}
%
\section{Basic estimates} \label{SB} 

We shall estimate an increase of fidelity $\mathcal{F}$ from above.
 We begin with a case of two-valued functions.
 For $\gamma\in(\alpha,\beta)$, we set
\[
U_0^\gamma (x) := \left\{
\begin{aligned}
	g(\alpha), &\quad x \in [\alpha, \gamma) \\	
	g(\beta), &\quad x \in [\gamma, \beta].
\end{aligned}
\right.
\]
The fidelity of $U_0^\gamma$ on $(\alpha,\beta)$ is denoted by $\lambda F(\gamma)/2$, i.e.,
\[
	F(\gamma) := \int_\alpha^\beta |U_0^\gamma - g|^2\, dx.
\]
\begin{prop} \label{PConMo}
Assume that $g\in C[\alpha,\beta]$ is non-decreasing and that $U_0^\gamma(\alpha)=g(\alpha)$, $U_0^\gamma(\beta)=g(\beta)$ with $\rho=U_0^\gamma(\beta) -U_0^\gamma(\alpha)>0$.
 If $\gamma\in[\alpha,\beta]$ satisfies $F(\alpha+0)\geq F(\gamma)$, then $g(x)\le \left(g(\alpha)+g(\beta)\right)/2$ for $x\in[\alpha,\gamma]$, which yields
\[
	\int_\alpha^\gamma g(x)\, dx
	\le \frac{g(\alpha)+g(\beta)}{2} (\gamma-\alpha).
\]
\end{prop}
This easily follows from the next observation.
\begin{prop} \label{PMF}
Assume that $g\in C[\alpha,\beta]$ is non-decreasing.
 Then $F(\gamma)$ is minimized if and only if $g(\gamma)=\left(g(\alpha)+g(\beta)\right)/2$.
\end{prop}
\begin{proof}
A direct calculation shows that
\begin{align*}
	F(\gamma) &= \int_\alpha^\beta |U_0^\gamma-g|^2 \; dx \\
	&= \int_\alpha^\gamma \left|g(x)-g(\alpha)\right|^2 dx
	+ \int_\gamma^\beta \left|g(\beta)-g(x)\right|^2 dx.
\end{align*}
Then
\begin{align*}
	F'(\gamma) &= \left|g(\gamma) - g(\alpha)\right|^2
	- \left|g(\beta) - g(\gamma)\right|^2 \\
	&= \left(2g(\gamma) - g(\alpha) - g(\beta)\right)
	\left(g(\beta) - g(\alpha)\right).
\end{align*}
Thus, $F$ takes its only minimum at $\gamma$ satisfying
\[
	g(\gamma) = \frac{g(\alpha)+g(\beta)}{2}.
\]
\end{proof}

We give a few explicit estimate of $TV_{Kg}$ for a monotone function $g$ to give a proof for Theorem~\ref{TMainMon}.

We recall $U_0^\gamma$ defined on $[\alpha,\beta]$.
 We shall simply denote $U_0^\gamma$ by $U_0$ if $g(\gamma)=\left(g(\alpha)+g(\beta)\right)/2$.
 Note that such $\gamma$ always exists since $g$ is continuous.
 It is unique if $g$ is increasing.
 In general, it is not unique and we choose one of such $\gamma$.
 We next consider a non-decreasing piecewise constant function with two jumps.
 We set $\rho=g(\beta)-g(\alpha)$ and $\delta\in(0,1)$ and define
\[
	U_\delta(x) = \left\{
\begin{array}{ll}
	g(\alpha), & x \in [\alpha, x_1) \\
	g(\alpha)+\delta\rho, & x \in [x_1, x_2) \\
	g(\beta), & x \in [x_2, \beta)
\end{array}
\right.
\]
for $x_1<x_2$ with $x_1,x_2\in(\alpha,\beta)$.
  By Proposition~\ref{PMF}, $\mathcal{F}(U_\delta)$ is minimized by taking $x_1=x_1^*$, $x_2=x_2^*$, where $x_1^*$ and $x_2^*$ are defined by
\[
	g(x_1^*) = g(\alpha) + \delta\rho/2, \quad
	g(x_2^*) = \left( g(\beta) + g(\alpha) + \delta\rho\right)/2 
	= g(\beta) - (1-\delta)\rho/2, 
\]
see Figure~\ref{FUD}.
\begin{figure}[tb]
\centering 
\includegraphics[keepaspectratio, scale=0.25]{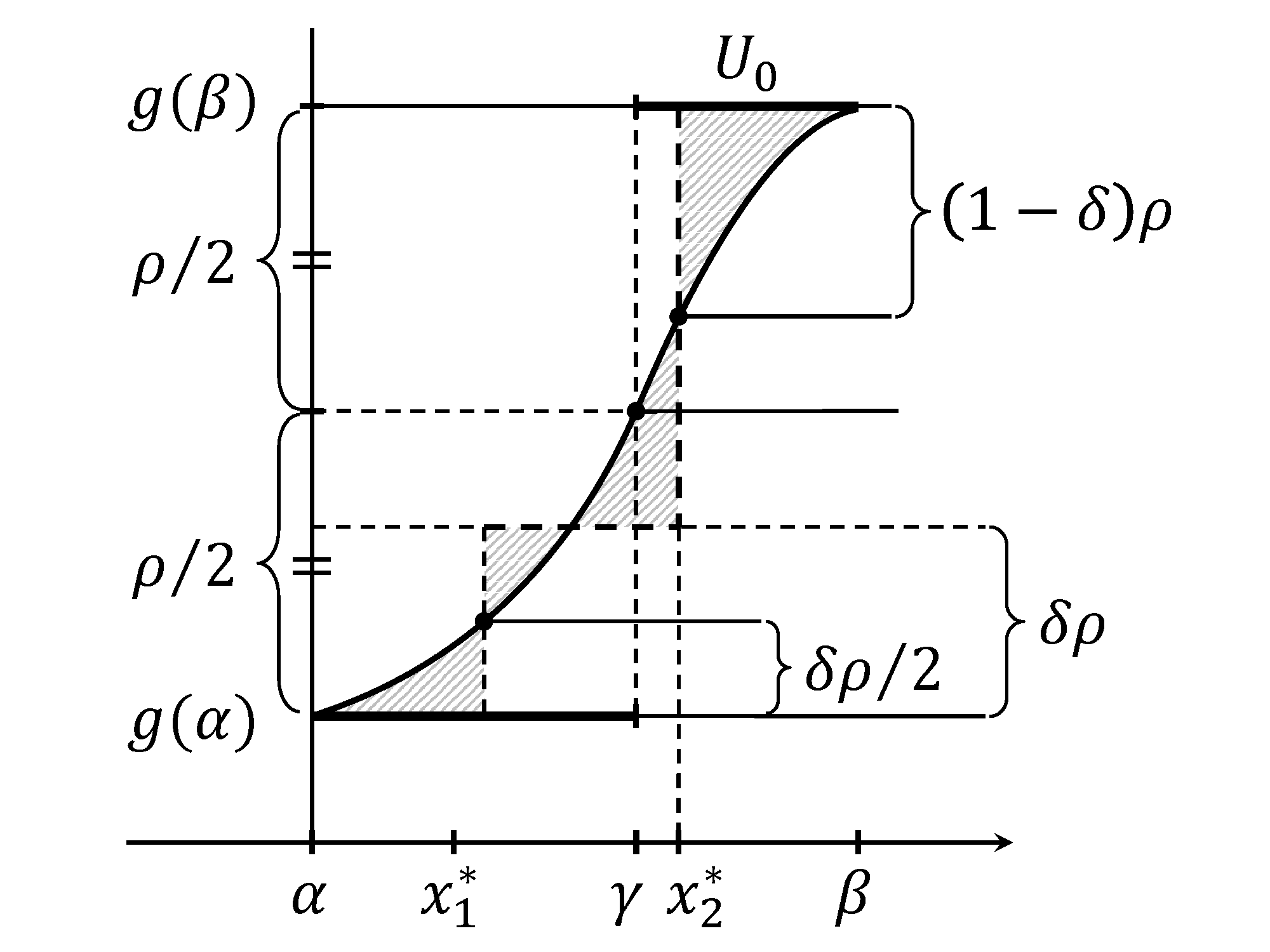}
\caption{$U_0$ and $U_\delta$\label{FUD}} 
\end{figure}
 Again $x_1^*$, $x_2^*$ may not be unique.
 We make an arbitrary choice of each.

We shall estimate
\[
	TV_{Kg}(U_\delta) - TV_{Kg}(U_0).
\]
This is done in \cite[Lemma 4.7]{GKKOS} since \cite[(4.1)]{GKKOS} is automatically fulfilled by Proposition~\ref{PConMo}.
 However, since the proof is clearer than that of general $g$ and a constant is improved, we give a full proof.
\begin{lemma} \label{LFid}
Assume that $g\in C[\alpha,\beta]$ is non-decreasing.
 Let $v$ be a non-decreasing piecewise constant function with three values $g(\alpha)$, $g(\alpha)+\delta\rho$, $g(\beta)$ for $\delta\in(0,1)$, where $\rho=g(\beta)-g(\alpha)$.
 Then
\[
	\int_\alpha^\beta (U_0-g)^2\; dx - \int_\alpha^\beta(v-g)^2 \;d x
	\leq \delta(1-\delta)\rho^2 (\beta-\alpha).
\]
\end{lemma}
\begin{proof}
Since $\mathcal{F}(v)\ge\mathcal{F}(U_\delta)$, it suffices to prove the inequality for $v=U_\delta$ by choosing $x_1=x_1^*$, $x_2=x_2^*$.
 Since $U_\delta=U_0$ for $x\in[\alpha,x_1^*]$ or $x\in[x_2^*,\beta]$, we observe that
\begin{align*}
	&\int_\alpha^\beta (U_0-g)^2\, dx 
	- \int_\alpha^\beta (U_\delta-g)^2\, dx
	=\int_{x_1^*}^{x_2^*} (U_0 + U_\delta - 2g) (U_0-U_\delta)\, dx \\
	&=2 \int_{x_1^*}^\gamma \left( g - \frac{U_0+U_\delta}{2} \right) (U_\delta-U_0)\, dx
	+ 2 \int_\gamma^{x_2^*} \left( \frac{U_0+U_\delta}{2} - g \right) (U_0-U_\delta)\, dx, \\
	&= I + I\!I.
\end{align*}
Since $U_\delta-U_0=\delta\rho$ and $(U_0+U_\delta)/2=g(x_1^*)$ on $(x_1^*,\gamma)$, we see that
\[
	I = 2\delta\rho \int_{x_1^*}^\gamma \left( g(x) - g(x_1^*) \right) dx
	\le 2\delta\rho \left( g(\gamma) - g(x_1^*) \right) (\gamma-x_1^*).
\]
Since
\[
	g(\gamma) - g(x_1^*) 
	= \frac{g(\alpha) + g(\beta)}{2} - \left( g(\alpha) + \frac\delta2 \rho \right)
	=\frac{1-\delta}{2} \rho,
\]
we now conclude that
\[
	I \le 2 \delta\rho \frac{1-\delta}{2} \rho (\gamma-x_1^*).
\]
Similarly, we estimate
\begin{align*}
	I\!I &= 2 \int_\gamma^{x_2^*} \left( g(x_2^*)-g(x)\right) dx
	\le 2(1-\delta)\rho \left( g(x_2^*)-g(\gamma)\right) (x_2^*-\gamma) \\
	&= 2(1-\delta) \rho \frac\delta2 \rho (x_2^*-\gamma).
\end{align*}
We thus obtain
\begin{align*}
	I + I\!I &\le \delta(1-\delta)\rho^2 (\gamma-x_1)
	+ \delta(1-\delta) \rho^2 (x_2-\gamma) \\
	&= \delta(1-\delta)\rho^2 (x_2 - x_1)
	\le \delta(1-\delta) \rho^2 (\beta-\alpha),
\end{align*}
which is the desired inequality.
\end{proof}

We are now in a position to estimate $TV_{Kg}(U_\delta)-TV_{Kg}(U_0)$.
 We set
\begin{multline*}
	X_\delta = \Big\{\, v \bigm|
	v\ \text{is a non-decreasing piecewise constant function} \\
	\text{with three facets and}\ 
	v(\alpha)=g(\alpha),\ v(\beta)=g(\beta). \ \text{Moreover,} \\
	\text{the value on the middle facet equals}\ 
	g(\alpha)+\delta\rho\, \Big\}.
\end{multline*}
\begin{lemma} \label{LND}
Assume that $K$ satisfies (K1) and (K2).
 Assume that $g\in C[\alpha,\beta]$ is non-decreasing.
 If $g(\beta)-g(\alpha)\leq M$ and $C_*:=C_M-(\beta-\alpha)\lambda/2>0$, then
\[
	\inf_{v\in X_\delta} TV_{Kg}(v) \geq TV_{Kg}(U_0)
	+ C_* \delta(1-\delta)\rho^2
\]
for $\delta\in(0,1)$, where $\rho=g(\beta)-g(\alpha)$.
\end{lemma}
\begin{proof}
Since the value $TV_K(v)$ is the same for all $v\in X_\delta$, we observe that
\[
	\inf_{v\in X_\delta} TV_{Kg}(v) \geq TV_{Kg}(U_\delta)
\]
with $x_1=x_1^*$, $x_2=x_2^*$.
 For $\rho=g(\beta)-g(\alpha)$, we set $\rho_1=\delta\rho$, $\rho_2=(1-\delta)\rho$ so that $\rho_1+\rho_2=\rho$.
 We observe that
\[
	TV_K(U_\delta) = K(\rho_1) + K(\rho_2), \quad
	TV_K(U_0) = K(\rho).
\]
By (K2), we observe that
\[
	TV_K(U_\delta) \geq TV_K(U_0) + C_M \rho_1 \rho_2
\]
for $\rho\leq M$.
 By Lemma~\ref{LFid}, we conclude that
\begin{align*}
	TV_{Kg}(U_\delta) &= TV_K(U_\delta) + \mathcal{F}(U_\delta) \\
	&\geq TV_K(U_0) + C_M \rho_1 \rho_2 + \mathcal{F}(U_0) 
	- \frac{\lambda}{2} \delta(1 - \delta)\rho^2 (\beta - \alpha) \\
	&= TV_{Kg}(U_0) + \delta(1 - \delta)\rho^2 \left(C_M - (\beta - \alpha) \lambda/2 \right).
\end{align*}
Thus
\[
	TV_{Kg}(U_\delta) - TV_{Kg}(U_0) \geq C_* \delta(1-\delta)\rho^2
\]
provided that $(\beta-\alpha)\lambda/2<C_M$, i.e., $C_*>0$.
The proof is now complete.
\end{proof}
%
\section{Quantitative estimates} \label{SQ} 

We shall estimate $TV_{Kg}(u)$ for a monotone function $u$ when $g$ is monotone.
\begin{lemma} \label{LEBJ}
Assume that $K$ satisfies (K1), (K2) and (K3).
 Assume that $g\in C[\alpha,\beta]$ is non-decreasing.
 Let $M$ be a constant such that $\rho=g(\beta)-g(\alpha)\leq M$.
 Let $U_0$ be a minimizer of $TV_{Kg}$ among piecewise constant functions with one jump in $(\alpha,\beta)$ satisfying $U(\alpha)=g(\alpha)$, $U(\beta)=g(\beta)$.
 Assume that $C_*=C_M-\lambda(\beta-\alpha)/2>0$.
 Let $v$ be a non-decreasing function on $[\alpha,\beta]$ which is continuous at $\alpha$ and $\beta$ and $v(\alpha)=g(\alpha)$, $v(\beta)=g(\beta)$.
 Then
\[
	TV_{Kg}(v) \geq \frac{C_*}{2} \left( \left(\sum_{z_i\in J_v} \rho_i \right)^2
	- \sum_{z_i\in J_v} \rho_i^2 + \left(\rho-\sum_{z_i\in J_v} \rho_i \right)^2 \right)
	+ TV_{Kg}(U_0),
\] 
where $\rho_i=v(z_i+0)-v(z_i-0)>0$ and $J_v=\{z_i\}_{i=1}^\infty$ ($\subset(\alpha,\beta)$) is the set of jump discontinuities.
 The non-negative term
\[
	\left( \sum_{z_i\in J_v} \rho_i \right)^2
	- \sum_{z_i\in J_v} \rho_i^2 + \left(\rho-\sum_{z_i\in J_v} \rho_i \right)^2
\]
vanishes if and only if $\rho=\rho_{i_0}$ with some $i_0$.
\end{lemma}

Since $TV_{Kg}(U_0)=K\left(\left|v(\beta-0)-v(\alpha+0)\right|\right)$, this lemma is a precise form of Lemma~\ref{LWEBJ}.

We begin with estimates of lengths of converging intervals.
\begin{lemma} \label{LCIn}
Let $I$ be a bounded open interval.
 For $m\in\mathbb{N}$ let $\lambda_i^m\in I$ ($1\le i\le k_m$) be taken such that
\[
    \lambda_i^m < \lambda_{i+1}^m
    \quad\text{for}\quad
    i=1, \ldots, k_m - 1
\]
and set $\rho_i^m=\lambda_{i+1}^m-\lambda_i^m$ ($i=1,\ldots,k_m-1$).
 Take $\rho_{k_m}^m>0$ such that $\lambda_{k_m}^m+\rho_{k_m}^m\in I$.
 For an at most countable set $\Lambda\subset I$, assume that
\[
    J := \bigcup_{\lambda\in\Lambda}I_\lambda (\rho_\lambda) \subset I
    \quad\text{with}\quad \rho_\lambda > 0
\]
is a disjoint union.
 Assume that for $\lambda\in\Lambda$ there exists a sequence $\{i_m^\lambda\}_{m=1}^\infty$ such that
\[
    I_{\lambda_{i_m^\lambda}^m} (\rho_{i_m^\lambda}^m )
    \to I_\lambda(\rho_\lambda)
    \quad\text{as}\quad m\to\infty.
\]
Let
\[
    Z_m = \left\{ i \in \mathbb{N} \bigm|
    i=i_m^\lambda\ \text{for some}\ \lambda\in\Lambda\ \text{and}\ 1\le i\le k_m \right\}.
\]
Then
\[
    \lim_{m\to\infty} \sum_{i\in Z_m} \rho_i^m = \sum_{\lambda\in\Lambda} \rho_\lambda.
\]
\end{lemma}
\begin{proof}
In general, $\rho_\lambda^m\to\rho_\lambda$ may not imply $\sum_{\lambda\in\Lambda}\rho_\lambda^m\to\sum_{\lambda\in\Lambda}\rho_\lambda$ as $m\to\infty$.
 We need to take geometry into account.

We consider an $\varepsilon$-neighborhood $J^\varepsilon$ of $J$, i.e.,
\[
    J^\varepsilon = \left\{ x \in I \bigm|
    |x-y|<\varepsilon\ \text{for all}\ y\in J \right\}.
\]
Evidently, the Lebesgue measure $|J^\varepsilon|$ of $J^\varepsilon$ converges to that of $J$, i.e.,
\[
    \lim_{\varepsilon\downarrow0} |J^\varepsilon| = |J|.
\]
Since $J^\varepsilon$ is open, it is at most a countable disjoint union of open intervals.
 However, the minimum length of such open interval is greater than $\varepsilon$ so $J^\varepsilon$ must be a finite disjoint union of open intervals, i.e.,
\[
    J^\varepsilon = \bigcup_{\ell=1}^{n_\varepsilon} J_\ell^\varepsilon.
\]
We decompose $\Lambda$ by $\Lambda=\bigcup_{\ell=1}^{n_\varepsilon}\Lambda_\ell^\varepsilon$ with
\[
    \Lambda_\ell^\varepsilon = \left\{ \lambda \in \Lambda \bigm|
    I_\lambda(\rho_\lambda) \subset J_\ell^\varepsilon \right\}.
\]
We also decompose $Z_m$ by $Z_m=\bigcup_{\ell=1}^{n_\varepsilon}Z_{m,\ell}^\varepsilon$ with
\[
    Z_{m,\ell}^\varepsilon = \left\{ i \in Z_m \bigm|
    i=i_m^\lambda \ \text{for some}\  \lambda \in \Lambda_\ell^\varepsilon \right\}.
\]
Since all $I_\lambda(\rho_\lambda)$ for $\lambda\in\Lambda_\ell^\varepsilon$ is contained in $J_\ell^\varepsilon$, we have
\[
    \sum_{i\in Z_{m,\ell}^\varepsilon} \rho_i^m
    \le |J_\ell^\varepsilon|, \quad
    1 \le \ell \le n_\varepsilon
\]
for sufficiently large $m$, say $m\ge m_0(\varepsilon)$.
 Thus
\[
    \varlimsup_{m\to\infty} \sum_{i\in Z_m} \rho_i^m
    = \varlimsup_{m\to\infty} \sum_{\ell=1}^{n_\varepsilon} \sum_{i\in Z_{m,\ell}^\varepsilon} \rho_i^m
    \le \sum_{\ell=1}^{n_\varepsilon} |J_\ell^\varepsilon| = |J^\varepsilon|.
\]
Sending $\varepsilon\downarrow0$ yields
\[
    \varlimsup_{m\to\infty} \sum_{i\in Z_m} \rho_i^m
    \le |J| = \sum_{\lambda\in\Lambda} \rho_\lambda.
\]
Since Fatou's lemma implies that
\[
    \varliminf_{m\to\infty} \sum_{i\in Z_m} \rho_i^m
    \ge \sum_{\lambda\in\Lambda} \rho_\lambda,
\]
the proof of Lemma~\ref{LCIn} is now complete.
\end{proof}

For later purpose, we prepare an elementary property of multiple of two sequences.
\begin{prop} \label{PSec}
Let $\{\rho_j\}_{j=1}^\infty$ be a sequence of positive numbers.
 Let $\{\rho_i^m\}_{i=1}^{k_m}$ be a set of non-negative numbers.
 Assume that for each $i$ there is $i_m^j$ such that $\rho_{i_m^j}^m\to\rho_j$ and
\[
    \sup \left\{ \rho_i^m \mid
    i\neq i_m^j\ \text{for any}\ j,\ 1
    \le i\le k_m \right\} \to 0
\]
$m\to\infty$.
 Assume that
\[
    s_\infty := \sum_{j=1}^\infty \rho_j
    \quad\text{and}\quad s:= \lim_{m\to\infty} \sum_{i=1}^{k_m}\rho_i^m
\]
exist and
\[
    \lim_{m\to\infty} \sum_{j=1}^\infty \rho_{i_m^j}^m = s_\infty\ \text{(}\le s\text{)}
\]
with interpretation that $\rho_i^m=0$ for $i>k_m$.
 Then
\begin{align*}
    \varliminf_{m\to\infty} \sum_{1\le i<\ell\le k_m} \rho_i^m \rho_\ell^m
    &\ge \sum_{1\le j<j'<\infty} \rho_j \rho_{j'} + (s-s_\infty)^2/2 \\
    &= \left(s_\infty^2 - \sum_{j=1}^\infty \rho_j^2 + (s-s_\infty)^2 \right)\biggm/ 2.
\end{align*}
\end{prop}
\begin{proof}
As in the proof of Lemma~\ref{LCIn}, we set
\[
    Z_m = \left\{ i \in \mathbb{N} \bigm|
    i=i_m^j\ \text{for some}\ j\ \text{and}\ 1\le i\le k_m \right\}.
\]
Then
\[
    \sum_{1\le i<\ell\le k_m} \rho_i^m \rho_\ell^m
    \ge \sum_{\substack{1\le i<\ell\le k_m\\ i,\ell\in Z_m}} \rho_i^m \rho_\ell^m
    + \sum_{\substack{1\le i<\ell\le k_m\\ i,\ell\in Z_m^c}} \rho_i^m \rho_\ell^m
    = I + I\!I.
\]
By Fatou's lemma
\[
    \varlimsup_{m\to\infty} I 
    \ge \sum_{ 1\le j<j'<\infty} \rho_j \rho_{j'}.
\]
For $i\in Z_m^c$, $\rho_i^m\to0$ uniformly so
\[
    \sum_{\substack{i=1 \\ i\in Z_m^c}}^{k_m} (\rho_i^m)^2
    \le \left( \max_{\substack{1\le i<k_m\\ i\in Z_m^c}} \rho_i^m \right) \sum_{i\in Z_m^c} \rho_i^m \to 0
\]
since $\sum_{i\in Z_m^c} \rho_i^m \to s-s_\infty$.
 Since
\[
    \sum_{\substack{1\le i<\ell\le k_m \\ i,\ell\in Z_m^c}} \rho_i^m \rho_\ell^m
    = \left( s_m^2 - \sum_{\substack{i=1 \\ i\in Z_m^c}}^{k_m} (\rho_i^m)^2\right)\Biggm/2,\quad
    s_m = \sum_{\substack{i=1 \\ i\in Z_m^c}}^{k_m} \rho_i^m,
\]
we thus observe that
\[
    \lim_{m\to\infty} I\!I = \left(\lim_{m\to\infty} s_m^2 - 0\right) \Bigm/2
    = (s-s_\infty)^2/2.
\]
Combining estimates for $I$ and $I\!I$, we now obtain
\[
    \varliminf_{m\to\infty} \sum_{1\le i<\ell\le k_m} \rho_i^m \rho_\ell^m
    \ge \sum_{1\le i<j<\infty} \rho_i \rho_j + (s-s_\infty)^2/2.
\]
The right-hand side equals
\[
    \left(s_\infty^2 - \sum_{j=1}^\infty \rho_j^2
    + (s-s_\infty)^2 \right) \biggm/2
\]
so the proof is now complete.
\end{proof}
\begin{proof}[Proof of Lemma~\ref{LEBJ}]
We may assume that $g(\beta)>g(\alpha)$.
 We approximate $v$ on $[\alpha,\beta]$ by a piecewise constant non-decreasing function $u_k$ with $m_k$ jumps by Corollary~\ref{CApMon}.
 Since $v$ is continuous at $\alpha$ and $\beta$, we may assume, by Remark~\ref{RApMon}, that $u_k(x)=g(\alpha)$ for $x$ close to $\alpha$ with $x>\alpha$ and that $u_k(x)=g(\beta)$ for $x$ close to $\beta$ with $x<\beta$.
 We denote jumps by $a_1<a_2<\cdots<a_{m-1}$ with $m=m_k+1$ and set $h_i=u_k(a_i+0)$ with convention that $a_0=\alpha$, $a_m=\beta$.
 We set
\[
	\rho_i = h_i - h_{i-1} \quad\text{for}\quad i = 1, \ldots, m_k
\]
which denotes the jump at each $a_i$.
 We fix $\rho_i$ and minimize $TV_{Kg}$ on $(\alpha, \beta)$.
 In other words, we minimize $\mathcal{F}$ by moving $a_i$'s.
 Let $\bar{u}_k$ be its minimizer.
 By Proposition~\ref{PMF}, facets (maximal intervals where $\bar{u}_k$ is a constant) of $\bar{u}_k$ consist of
\[
	[\alpha, x_1], [x_1, x_2], \ldots, [x_{m_k-1}, \beta]
\]
with $g(a_i)=\left(g(x_i)+g(x_{i+1})\right)/2$, $i=1,\ldots,m_k-1$; see Figure~\ref{FPrU} with $x_0=a_0,\ldots$, $x_{m-1}=a_m=\beta$.
\begin{figure}[tb]
\centering 
\includegraphics[keepaspectratio, scale=0.25]{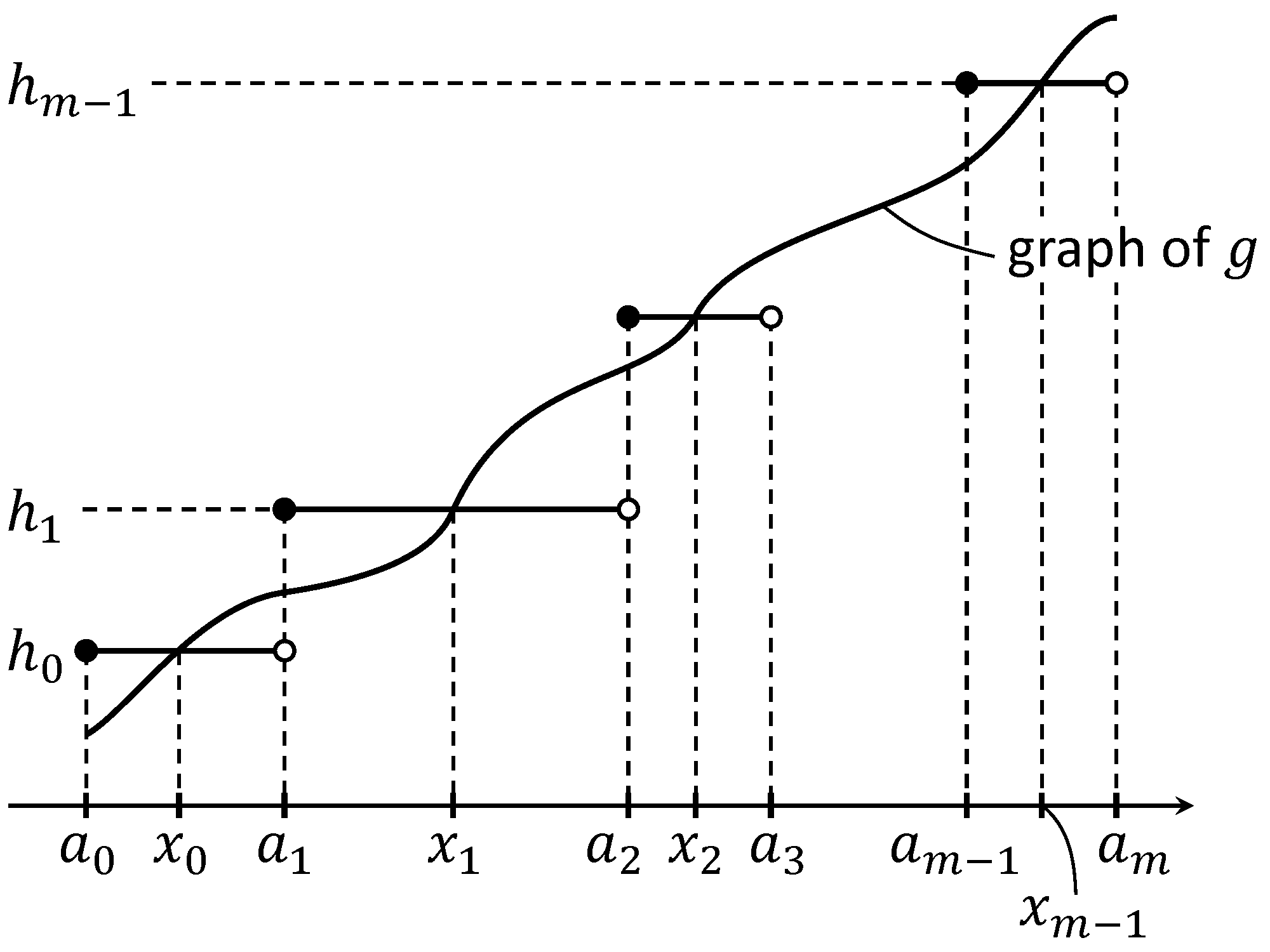}
\caption{the graphs of $g$ and $U$\label{FPrU}}
\end{figure}
 By this choice,
\[
	TV_{Kg}(u_k) \geq TV_{Kg}(\bar{u}_k).
\]

We shall estimate $TV_{Kg}(\bar{u}_k)$ from below as in Lemma~\ref{LND}.
 For $(\alpha,x_i)$,  let $V_i$ be a piecewise constant function on $[\alpha,x_i]$ with one jump at the point $y_i$ such that
\[
	g(y_i) = \left(g(x_i)+g(\alpha)\right)\bigm/2
	\quad\text{and}\quad V_i(\alpha) = g(\alpha), \quad
	V_i(x_i) = u_k(x_i).
\]
We set
\[
	W_i(x) = \left\{
\begin{array}{ll}
	V_i(x), & x \in [\alpha, x_i) \\
	\bar{u}_k(x), & x \in [x_i, \beta].
\end{array}
\right.
\]
We note that $W_{m_k}=V_{m_k}=U_0$.
 Since $C_*>0$, we argue as in Lemma~\ref{LND} and observe that
\begin{align*}
	TV_{Kg}(\bar{u}_k) &\geq TV_{Kg}(W_2) + C_* \rho_1 \rho_2 \\
	&\geq TV_{Kg}(W_3) + C_* \rho_3(\rho_1+\rho_2) + C_* \rho_1 \rho_2 \\
	&\cdots \\
	&\geq TV_{Kg}(V_{m_k}) + C_* \sum_{1\leq i<j\leq m_k}\rho_i \rho_j
	= TV_{Kg}(U_0) + C_* \sum_{1\leq i<j\leq m_k} \rho_i \rho_j.
\end{align*}

From now on, we write jumps of $u_k$ by $\rho_i^k$ instead of $\rho_i$.
 Our estimate for $TV_K(\bar{u}_k)$ yields
\[
	TV_{Kg}(u_k) \geq TV_{Kg}(U_0) + C_* \sum_{1\leq i<j\leq m_k}\rho_i^k \rho_j^k.
\]
We now apply Theorem~\ref{TCJ}.
 Then by Lemma~\ref{LCIn} we are able to apply Proposition~\ref{PSec} to get
\begin{align*}
	TV_{Kg}(v) &\geq TV_{Kg}(U)
    = \lim_{k\to\infty} TV_{Kg}(u_k) \ge TV_{Kg}(U_0)
    + C_* \varliminf_{k\to\infty} \sum_{1\le i<j\le m_k} \rho_i^k \rho_j^k \\
    &\ge TV_{Kg}(U_0) + \frac{C_*}{2} \left(s_\infty - \sum_{j=1}^\infty \rho_j^2 + (s-s_\infty)^2 \right)
\end{align*}
where $\{\rho_j\}$ be the set of all jumps of $U$ and $s_\infty:=\sum_{j=1}^\infty\rho_j$ ($\le g(\beta)-g(\alpha)$),
\[
    s := \lim_{k\to\infty} \sum_{i=1}^{m_k}\rho_i^k.
\]
Note that $s=\rho$ since $u_k$ approximates $v$.
 Since
\[
    s_\infty^2 - \sum_{j=1}^\infty \rho_j^2 \ge 0
    \quad\text{and}\quad(s-s_\infty)^2 \ge 0,
\]
the quantity
\[
    s_\infty^2 - \sum_{j=1}^\infty \rho_j^2 + (s-s_\infty)^2 = 0
\]
if and only if $s=s_\infty$ and $\sum_{j=1}^\infty \rho_j^2=\left(\sum_{j=1}^\infty \rho_j\right)^2$.
 The second identity implies that $\rho_j=0$ except one index $i_0$.
 This means that $v$ must have only one jump with size $\rho$ ($=s$) which equals $s_\infty$.
 The proof is now complete.
\end{proof}
\begin{proof}[Proof of Theorem~\ref{TMainMon}]
We may assume that $g$ is non-decreasing and not a constant.
 By \cite[Lemma 4.12]{GKKOS}, a minimizer $u$ is non-decreasing.
 By \cite[Lemma 3.2]{GKKOS}, the function $U$ is continuous on
\[ 
    C=\left\{x\in(a,b)\mid U(x)=g(x)\right\}.
\]
 Let $Q\subset\left[g(a),g(b)\right]$ be the set of $q$ such that
\[
    I_q = \left\{ x \mid g(x) = q \right\}
\]
is not a singleton.
 We set
\[
    C_\Gamma = C \setminus \bigcup_{q\in Q} I_q^*, \quad
    I_q^* = \left\{ x \in I_q \mid x> \inf I_q \right\}.
\]
Since $g$ is continuous, $C_\Gamma$ is not empty by \cite[Lemma 3.1]{GKKOS}.

If we prove that $C_\Gamma$ is a discrete set in $[a,b]$ so that it is a finite set, then, by \cite[Lemma 3.1]{GKKOS}, $U$ is a piecewise constant function.
 We shall prove that $C_\Gamma$ is a discrete set.
 By definition, if $x_1<x_2$ for $x_1,x_2\in C_\Gamma$, then $g(x_1)<g(x_2)$.
 Suppose that $C_\Gamma$ were not discrete.
 Then for any small $\varepsilon>0$, there would exist $x_1,x_2\in C_\Gamma$ with $x_1<x_2<x_1+\varepsilon$ such that the interval $(x_1,x_2)$ would contain infinity many elements of $C_\Gamma$.
 Since $U$ minimizes $TV_{Kg}$ on $(x_1,x_2)$ with $U(x_1)=g(x_1)$, $U(x_2)=g(x_2)$, applying Lemma~\ref{LEBJ} to $\alpha=x_1$, $\beta=x_2$ to conclude that $U$ must have at most one jump provided that $\varepsilon<2C_M/\lambda$.
 This yields a contradiction so we conclude that $C_\Gamma$ is discrete and $U$ is a non-decreasing piecewise constant function.
 The number of jumps can be estimated since the distance of two points in $C$ is at most $2C_M/\lambda$.
\end{proof}

As an application of approximation (Lemma~\ref{LApC}) we shall prove Corollary~\ref{CMain2}.

\begin{proof}[Proof of Corollary~\ref{CMain2}]
We approximate $g\in L^\infty(a,b)$ by $g_\ell\in C[a,b]$ such that $g_\ell\to g$ in $L^2(a,b)$ as $\ell\to\infty$ and
\[
    \essinf g \le g_\ell \le \esssup g
    \quad\text{on}\quad (a,b)
\]
for all $\ell\ge1$.
 Let $u\in BV(a,b)$ with $TV_{Kg}(u)<\infty$.
 By Lemma~\ref{LApC}, there is a sequence $\{u_\ell\}$ of piecewise constant functions such that $TV_K(u_\ell)\to TV_K(u)$ with $u_\ell\to u$ in $L^2(a,b)$ as $\ell\to\infty$.
 We thus observe that $\{u_\ell\}$ is a ``recovery" sequence in the sense that $TV_{Kg}(u_\ell)\to TV_{Kg}(u)$.
 Let $U_\ell$ be a minimizers of $TV_{Kg_\ell}$.
 By Theorem~\ref{TEJ}, the number of jump $m_\ell$ is bounded by
\[
    m_\ell \le \left[ (b-a)\lambda / A_M \right] + 1 =: m_* 
\]
with $M=\osc_{[a,b]}g$.
 By compactness of a bounded set in a finite dimensional spaces, $\{U_\ell\}$ has a convergent subspace, i.e., $U_\ell\to U$ with some $U$ in $L^2(a,b)$ (and also with respect to weak* topology of $BV$).
 Moreover, the limit $U$ is still piecewise constant function with at most $m_*$ jumps.
 By lower semicontinuity of $TV_K$ (\cite[Proposition 2.4]{GKKOS}), we see that
\[
    TV_{Kg}(U) \le \varliminf_{\ell\to\infty} TV_{Kg_\ell}(U_\ell)
    \le \varliminf_{\ell\to\infty} TV_{Kg_\ell}(u_\ell) = TV_{Kg}(u).
\]
We thus conclude that $U$ is a desired minimizer of $TV_{Kg}$ since $u\in BV(a,b)$ is an arbitrary element of $BV(a,b)$.
\end{proof}

In the rest of this this section, we shall prove a more general version of Theorem~\ref{TSJ}.
 As an application, we give an example that minimizers of $TV_{Kg}$ is not unique.
 For $K=K(\rho)$ and $c>0$, we set
\[
    Q_c(z) := K(c-z) + K(c+z)
    \quad\text{for}\quad
    z \in[-c,c].
\]
The main assumption is a kind of super linearity for the derivative $-Q'_c$.
\begin{enumerate}
\item[(Q)] The function $P=-Q_c$ is $C^1$ in $z\in[-c,c]$ and satisfies
\[
    P'(\mu z) < \mu P'(z)
    \quad\text{for}\quad
    z \in (0,c]
    \quad\text{and}\quad
    \mu \in (0,1).
\]
\end{enumerate}
\begin{prop} \label{PExj}
If
\[
    K(\rho) = \frac{\rho}{1+\rho\kappa}
\]
with $\kappa>0$, then $K$ fulfills (Q) for all $c>0$.
\end{prop}
\begin{proof}
This can be proved by a direct calculation.
 We observe that
\[
    K(\rho) = 1-\frac{\kappa}{1+\rho\kappa}
    = 1- \frac{1}{\frac{1}{\kappa}+\rho}.
\]
If we set $c'=c+1/\kappa$, then
\[
    Q_c(z) = 2 - \frac{1}{c'-z}- \frac{1}{c'+z}
\]
so that
\[
    Q'_c(z) = \frac{-1}{(c'-z)^2}+ \frac{1}{(c'+z)^2}
    = \frac{-(c'+z)^2 + (c'+z)^2}{(c'^2-z^2)^2}
    = \frac{-4c'z}{(c'^2-z^2)^2}.
\]
We now observe that
\[
    -Q'_c(\mu z) = \frac{4c'\mu z}{\left(c'^2-(\mu z)^2\right)^2}
    \le \frac{4c'\mu z}{(c'^2-z)^2} \le -Q'_c(z)
    \quad\text{for}\quad
    z \in [0,c),\ \mu \in|0,1)
\]
since $c'^2-(\mu z)^2\ge c'^2-z^2>0$.
 Thus, condition (Q) is fulfilled.
\end{proof}
\begin{thm} \label{TSJG}
Assume that $g$ is linear with $g'>0$ and that $K$ satisfies (Q).
 For $\Omega=(a,b)$, let $U$ be a non-decreasing piecewise constant minimizer of $TV_{Kg}$ which only jumps at $a<a_1<a_2<\cdots<a_m<b$.
 Then $a_{i+1}-a_i$ ($i=1,\ldots,m-1$) and the jump size $U(a_i+0)-U(a_i-0)$ ($i=1,\ldots,m$) are independent of $i$.
\end{thm}

We shall calculate $\int|u-g|^2\,dx$ with $g(x)=x$ for some typical step functions $u$.
 Let $\operatorname{sgn}$ denote the (left semicontinuous) signature function, i.e.,
\[
\operatorname{sgn} x = \left\{
\begin{aligned}
	-1, &\quad x \le 0, \\	
	1, &\quad x > 0.
\end{aligned}
\right.
\]
Then for $\sigma>0$
\[
    \int_{-\sigma}^\sigma |\sigma\operatorname{sgn}x-x|^2\,dx
    = \sigma^3 \int_{-1}^1 |\operatorname{sgn}y-y|^2\,dy
\]
by changing the variable $x=\sigma y$.
 Since
\[
    \int_{-1}^1 |\operatorname{sgn}y-y|^2\,dy
    = 2 \int_0^1 y^2 \,dy = \frac23,
\]
we end up with
\begin{equation} \label{EBas1}
    \int_{-\sigma}^\sigma |\sigma\operatorname{sgn}x-x|^2\,dx
    = \frac23 \sigma^3 = \frac{(2\sigma)^3}{12}.
\end{equation}
We next calculate $\int|u-x|^2\,dx$ for
\[
u^z(x) := \left\{
\begin{aligned}
	\sigma\operatorname{sgn}(x+\sigma')-\sigma', &\quad -1<x<z, \\	
	\sigma'\operatorname{sgn}(x-\sigma)+\sigma, &\quad z<x<1
\end{aligned}
\right.
\]
with $\sigma=(1+z)/2$, $\sigma'=(1-z)/2$, where $z\in(-1,1)$.
 This is a monotone step function with two jumps.
 It can be written as
\[
u^z(x) = \left\{
\begin{aligned}
	-1, &\quad -1<x\le-\sigma', \\	
	z, &\quad -\sigma'<x\le\sigma, \\
    1, &\quad \sigma<x<1;
\end{aligned}
\right.
\]
see Figure~\ref{FSgn}.
\begin{figure}[tb]
\centering 
\includegraphics[keepaspectratio, scale=0.27]{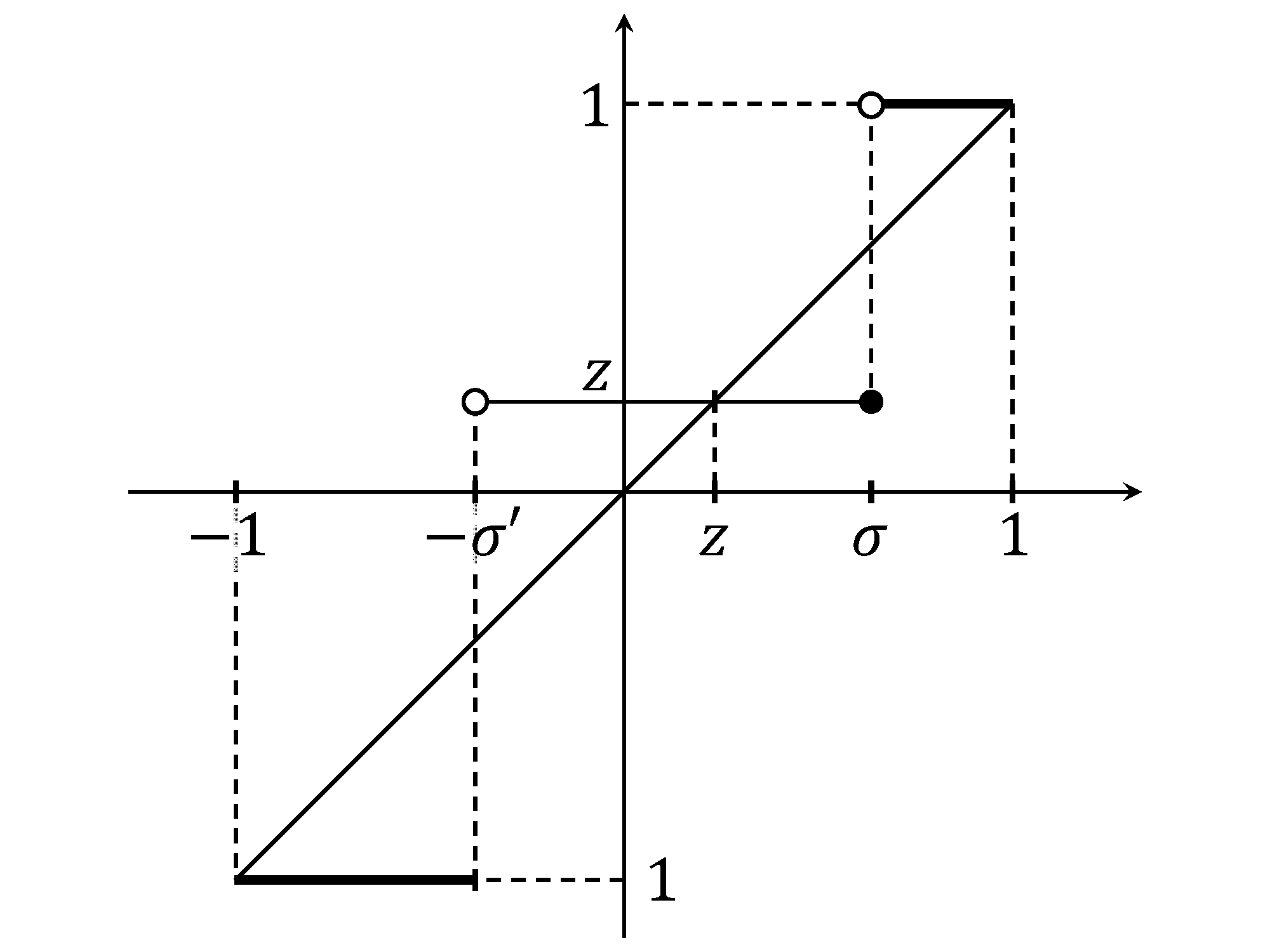}
\caption{the graph of $u^z$\label{FSgn}}
\end{figure}
\begin{lemma} \label{LSgn}
For $z\in(-1,1)$,
\[
    \int_{-1}^1 |u^z-x|^2 \,dx
    = \frac{z^2}{2} + \frac16.
\]
\end{lemma}
\begin{proof}
We proceed
\begin{align*}
    \int_{-1}^1 |u^z-x|^2 \,dx
    &= \int_{-1}^z |u^z-x|^2 \,dx
    + \int_z^1 |u^z-x|^2 \,dx, \\
    \int_{-1}^z |u^z-x|^2 \,dx
    &= \int_{-\sigma}^\sigma |\sigma\operatorname{sgn} x-x|^2 \,dx, \\
    \int_z^1 |u^z-x|^2 \,dx
    &= \int_{-\sigma'}^{\sigma'} |\sigma' \operatorname{sgn} x-x|^2 \,dx
\end{align*}
by translation.
 By \eqref{EBas1}, we now obtain
\[
    \int_{-1}^z |u^z-x|^2 \,dx
    = \frac{(1+z)^3}{12}, \quad
    \int_z^1 |u^z-x|^2 \,dx
    = \frac{(1-z)^3}{12}.
\]
We now conclude that
\[
    \int_{-1}^1 |u^z-x|^2 \,dx
    = \frac{1}{12} \left[(1+z)^3+(1-z)^3 \right]
    = \frac{1}{12} 2\cdot (1+3z^2)
    = \frac12 z^2 + \frac16.
\]
\end{proof}
\begin{proof}[Proof of Theorem~\ref{TSJG}]
It suffices to prove that
\[
    U(a_i+0) - U(a_i-0)
    = U(a_{i+1}+0) - U(a_{i+1}-0), \quad
    i=1,\ldots,m-1
\]
since $a_{i+1}-a_i=a_{i+2}-a_{i+1}$ ($i=1,\ldots,m-1$) follows from
\[
    \frac{U(a_i+0)+U(a_i-0)}{2}
    = g(a_i) \quad
    \text{(by Proposition~\ref{PMF})}
\]
and $g$ is linear.
 By restricting the functional $TV_{Kg}$ in $(a_i-\delta,a_{i+1}+\delta)$ for small $\delta>0$, we may assume that $m=2$.
 By translation and dilation we may assume $\Omega=(-a,a)$ with $a>1$, and $U(\pm1)=g(\pm1)$ and $g(x)=kx$, $k>0$.
  We change the depending variable $U$ by $ku$ to get
\[
    TV_K(U) = TV_{K_k}(u)
\]
with $K_k(\rho)=K(k\rho)$.
 We note that if $K$ satisfies (Q), so is $K_k$.
 Since
\[
    TV_{Kg}(U) = TV_K(U)
    + \frac\lambda2 \int_\Omega |U-kx|^2\,dx
    = TV_{K_k}(u) + \frac{\lambda k^2}{2} \int_\Omega |u-x|^2\,dx,
\]
we may assume that $g(x)=x$.
 By Proposition~\ref{PMF}, a minimizer (with two jumps) equals $u^z$ for some $z\in(-1,1)$, since $U(\pm1)=g(\pm1)=\pm1$.

We next observe that
\[
    E(z) := TV_{Kg}(u^z) = Q_1(z)
    + \frac\lambda2 \int_{-1}^1 |u^z-x|^2\,dx
    = Q_1(z) + \frac\lambda2 \left( \frac{z^2}{2}+\frac16 \right)
\]
by Lemma~\ref{LSgn}.
 By definition, $E(z)$ is an even function, i.e., $E(z)=E(-z)$ for $z\in[0,1]$.
 Setting $P=-Q_1$, we get
\[
    E'= \frac{dE}{dz} (z)
    = \frac\lambda2 z - P'(z).
\]
By (Q), only following three cases occur.
\begin{enumerate}
\item[(i)] $E'$ is positive on $[0,1]$;
\item[(i\hspace{-0.1em}i)] $E'$ is positive on $[0,\rho)$ for some $\rho\in(0,1)$ and negative for $(\rho,1]$;
\item[(i\hspace{-0.1em}i\hspace{-0.1em}i)] $E'$ is negative on $[0,1]$;
\end{enumerate}
see Figure~\ref{FCD}.
\begin{figure}[tb]
\centering 
\includegraphics[keepaspectratio, scale=0.22]{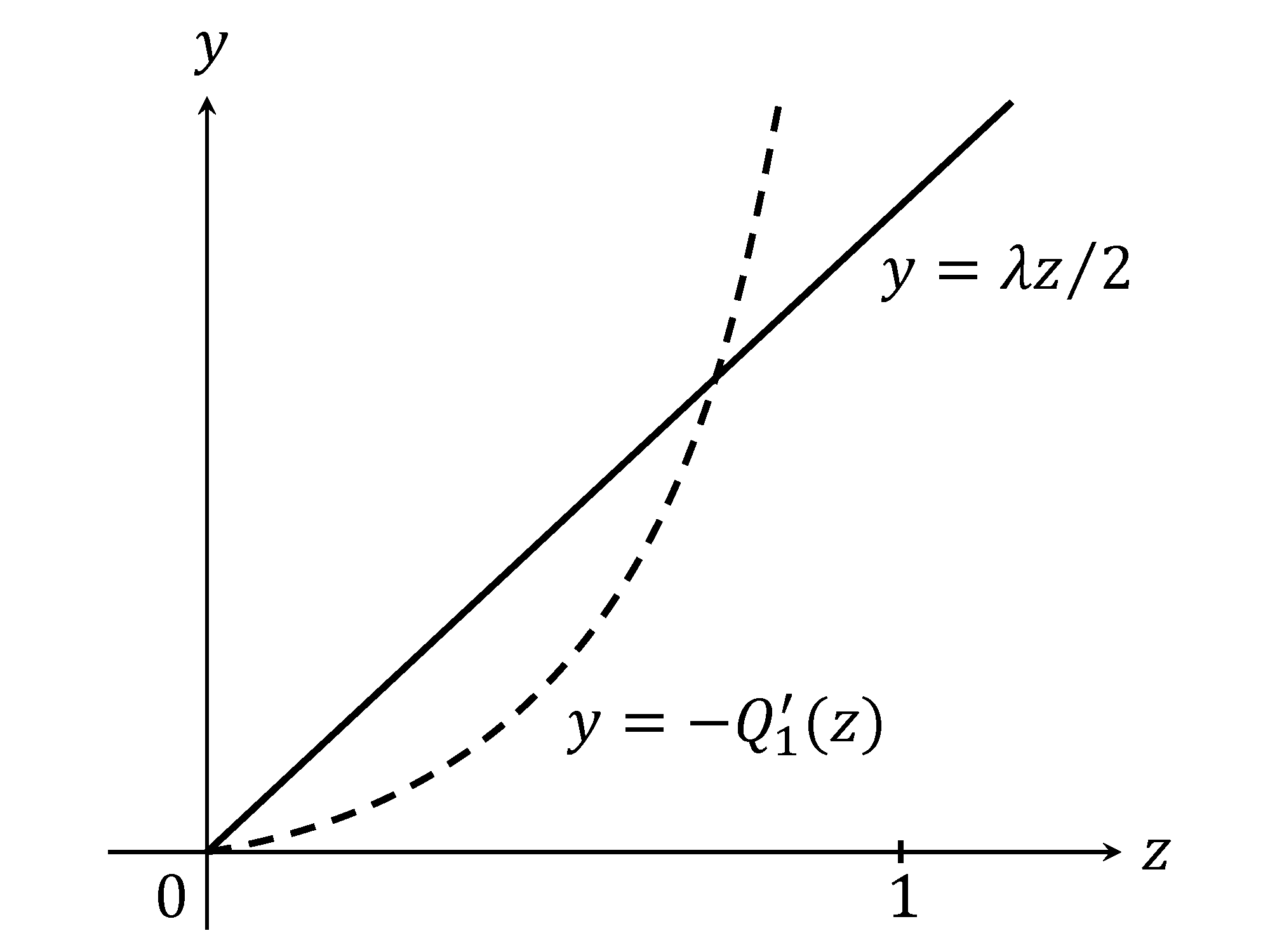}
\caption{the graph of $-Q'_1$\label{FCD}}
\end{figure}
Thus there is no chance that $E$ is minimized in $[0,1]$ except points $z=0$ and $z=1$.
 Since we have assumed that there are two jumps, the case $z=1$ should be excluded.
 Thus, in our setting the minimum value of $E$ is only attained at $z=0$, i.e., two jumps have the same size.
 The proof is now complete.
\end{proof}
\begin{proof}[Proof of Theorem~\ref{TSJ}]
By Proposition~\ref{PExj}, Theorem~\ref{TSJG} yields Theorem~\ref{TSJ}.
\end{proof}

As an application we prove that there is a chance that minimizers are not unique.

For a given $L>0$ and $m=1,2,\ldots$, we consider a step function
\begin{equation} \label{EStm}
    u_m(x) := kd \ \text{for}\ 
    x \in \left( \left(k-\frac12\right)d, \left(k+\frac12\right)d \right], \quad
    d=L/m, \quad
    k=0,1,2,\ldots,m
\end{equation}
on $(0,L)$ and calculate $TV_{Kg}(u_m)$ for $g(x)=x$.
 By Lemma~\ref{LSgn}, we observe that
\[
    TV_{Kg}(u_m)
    = K(d)m + \frac\lambda2 m \frac{d^3}{12}
    = L \left\{ \frac{K(d)}{d} + \frac\lambda2\frac{d^2}{12} \right\}
\]
by \eqref{EBas1}.
 If $K(\rho)=\rho/(1+\rho)$, then
\[
    \frac{K(\rho)}{\rho}
    = \frac{1}{\rho+1}.
\]
Thus
\begin{equation} \label{EESt}
    \frac1L TV_{Kg}(u_m)
    = \frac{1}{\frac{L}{m}+1} + \frac{\lambda}{24} \left(\frac{L}{m} \right)^2
    =: E(m).
\end{equation}
\begin{prop} \label{PCal}
Let $E$ be given by \eqref{EESt}.
 Then, there exists a unique value $\lambda$ such that there exists $m_0(\lambda)\in(1,2)$ satisfying
\begin{enumerate}
\item[(i)] $E(m)$ is decreasing in $\left(0,m_0(\lambda)\right]$;
\item[(i\hspace{-0.1em}i)] $E(m)$ is increasing in $\left[m_0(\lambda), \infty\right)$;
\item[(i\hspace{-0.1em}i\hspace{-0.1em}i)] $E(1)=E(2)$.
\end{enumerate}
Such $\lambda$ is explicitly written as
\[
	\lambda = \frac{2^5}{L(L+1)(L+2)}.
\]
\end{prop}
\begin{proof}
We set
\[
    f(d) := \frac{1}{d+1} + \frac{\lambda}{24} d^2
    = E(L/d), \quad d>0.
\]
The condition (i\hspace{-0.1em}i\hspace{-0.1em}i) is equivalent to
\begin{equation} \label{EDf}
    f(L) = f\left( \frac{L}{2} \right).
\end{equation}
Since $f$ is a strictly convex function, $f$ has a unique minimum on $(L/2,L)$ if \eqref{EDf} holds.
 For $E(m)$, this implies the existence of $m_0(\lambda)\in(1,2)$ if \eqref{EDf} is fulfilled.
 To complete the proof, it suffices to find $\lambda$ such that \eqref{EDf} holds.
 The condition \eqref{EDf} is of the form
\[
	\frac{1}{L+1} + \frac{\lambda}{24} L^2
	= \frac{1}{\frac{L}{2}+1} + \frac{\lambda}{24} \left( \frac{L}{2} \right)^2
\]
or
\[
	\frac{\lambda}{24} \left( 1-\frac14 \right) L^2
	= \frac{2}{L+2} - \frac{1}{L+1}.
\]
Since the right-hand side equals
\[
	\frac{L}{(L+2)(L+1)},
\]
we end up with
\[
	\lambda = 24 \cdot \frac43 \frac{1}{L(L+1)(L+2)}.
\]
Thus $\lambda$ is uniquely determined and for such value, \eqref{EDf} is fulfilled.
\end{proof}
\begin{remark} \label{RNU}
From the proof, it is easy to see we can take any interval $(\ell,\ell+\alpha)$, $\alpha>0$ instead of $(1,2)$.
\end{remark}
\begin{proof}[Proof of Corollary~\ref{CSJ}]
We may assume that minimizer is of the form \eqref{EStm} by Theorem~\ref{TSJ}.
 The number $m$ is taken so that $TV_{Kg}(u_m)$ is minimized.
 However, if $\lambda$ is chosen as in Proposition~\ref{PCal}, then both $u_1$ and $u_2$ minimize $TV_{Kg}$.
 Thus the minimizer is not unique.
\end{proof}

\section{Numerical Experiments} \label{SNE}

In this section, we present numerical experiments to compare the segmentation results of the proposed Kobayashi-Warren-Carter (KWC) type energy with the original Rudin-Osher-Fatemi (ROF) energy and the Mumford-Shah energy.
Since the direct numerical computation of $L^2$-gradient flows for non-convex energies involving jump sets is highly challenging, we employ their phase-field approximations for the Mumford-Shah and the proposed models.
Specifically, we compute the $L^2$-gradient flows for the following three models in a one-dimensional setting $\Omega = (0,1)$:

\begin{enumerate}
\item \textbf{ROF (TV) model}: We directly compute the gradient flow of the standard ROF energy:
\[
    E_{ROF}(u) = \sigma \int_\Omega |Du|
    + \mathcal{F}(u), \quad
    \mathcal{F}(u) = \frac{\lambda}{2} \int_\Omega |u-g|^2 \,dx.
\] 
\item \textbf{AT model (Phase-field approximation of Mumford-Shah)}: We use the Ambrosio-Tortorelli (AT) approximation:
\[
    E_{AT}^{\varepsilon}(u,v) = \int_\Omega \left( \sigma v^2 |u'|^2 + \frac{\varepsilon}{2} |v'|^2 + \frac{1}{2\varepsilon} (v-1)^2 \right) dx 
    + \mathcal{F}(u).
\]  
\item \textbf{KWC model (Phase-field approximation of $TV_{Kg}$)}: Instead of directly solving the singular limit model $TV_{Kg}(u)$ with the non-convex jump energy $K(\rho)$, we compute the gradient flow of the original KWC energy (\cite{GOSU} and \cite{GKKOSU}), which serves as its phase-field approximation:
\[
    E_{KWC}^{\varepsilon}(u,v) = \sigma \int_\Omega v^2 |Du| + \int_\Omega \left( \frac{\varepsilon}{2} |v'|^2 + \frac{1}{2\varepsilon} (v-1)^2 \right) dx
    + \mathcal{F}(u).
\]
\end{enumerate}

The numerical scheme is based on the Chambolle-Pock primal-dual algorithm for the non-smooth total variation term $|Du|$ without any regularization, while the phase-field variable $v$ is updated via an alternating implicit scheme. 
To strictly match the theoretical assumptions in Section \ref{SQ} (in particular, the proof of Theorem \ref{TSJG} assuming $U(\pm 1) = g(\pm 1)$), we impose the Dirichlet boundary condition for $u$ such that $u(0)=g(0)$ and $u(1)=g(1)$, while maintaining the homogeneous Neumann boundary condition for $v$.

Unless otherwise specified, we use the following common parameters for the simulations: the number of spatial grid points $N=1000$ (which corresponds to a grid size of $h=1/999$), the time step size $\Delta t = 0.01$, the total variation weight $\sigma = 1.0$, and the phase-field parameter $\varepsilon = 0.005$. The computations are carried out up to a sufficiently large time (e.g., $T=100$) to ensure that the solutions have completely reached their steady states.

\subsection{Verification of Theorem \ref{TSJ}: Linear Data}
First, we consider the linear monotone data $g(x) = x$ to verify Theorem \ref{TSJ}. 
Because the KWC energy is highly non-convex, directly computing the $L^2$-gradient flow with the naive initial condition $u(x,0) = g(x)$ often leads to local minima. 
Indeed, as illustrated in Figure \ref{Fig:Exp1}(a), the gradient flow gets trapped in a local minimum, resulting in a piecewise constant function with irregular jump sizes and intervals.

Therefore, to properly verify the properties of the global minimizer stated in Theorem \ref{TSJ}, we employ the theoretically predicted piecewise constant functions defined in \eqref{EStm} as initial conditions.
In the phase-field computation, to prevent the initial sharp jumps of $u$ from collapsing under the strong total variation penalty, we apply a pre-relaxation step: we first update $v$ to its steady state while keeping $u$ fixed, forming proper phase-field profiles (damage zones) at the jump locations, and then start the coupled gradient flow.
We confirm that these theoretical states are stably preserved as steady states.

Figure \ref{Fig:Exp1}(b) shows the steady-state solutions of the phase-field KWC model for varying optimal parameter sets $(\lambda, m)$ calculated from the theoretical energy evaluation.
As proved in Theorem \ref{TSJ}, the global minimizer $U$ of the KWC type energy is a non-decreasing piecewise constant function. 
Remarkably, the numerical result strictly reproduces the exact shape of the global minimizer: the distance between jump points $a_{i+1}-a_i$ and the jump sizes $U(a_i+0)-U(a_i-0)$ are perfectly uniform, and the lengths of the steps at the boundaries are exactly half of the interior steps.

\begin{figure}[tb]
\centering 
\includegraphics[width=\linewidth]{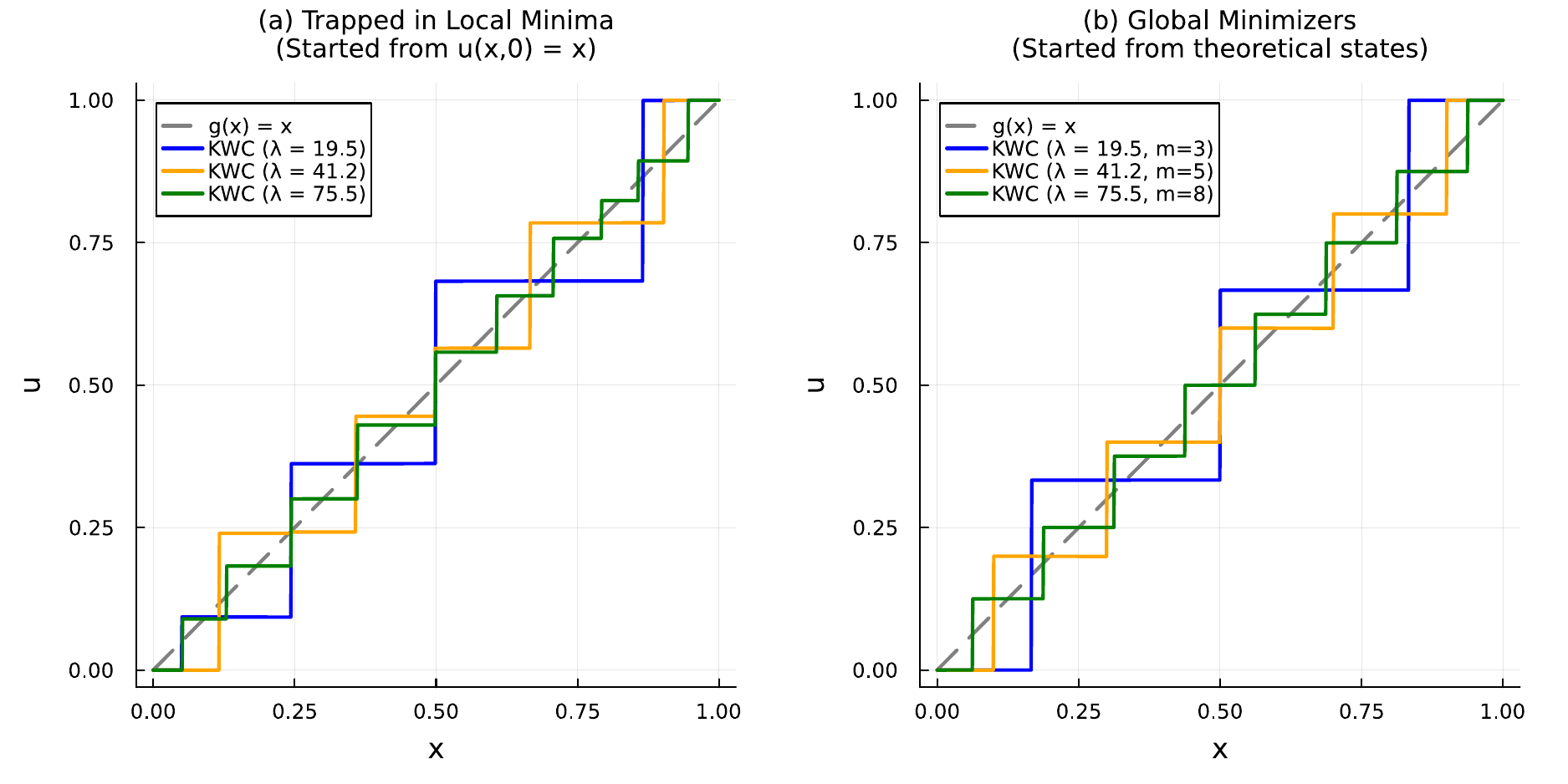}
\caption{Numerical results for linear data $g(x)=x$. (a) The KWC gradient flow starting from the naive initial condition $u(x,0)=g(x)$ gets trapped in a local minimum with irregular jumps. (b) By starting from the theoretically predicted states and applying Dirichlet boundary conditions, the phase-field KWC model perfectly captures the equal jump intervals and sizes of the true global minimizer.}
\label{Fig:Exp1}
\end{figure}

\subsection{Verification of Corollary \ref{CSJ}: Non-uniqueness of Minimizers}
Next, we demonstrate the non-uniqueness of the minimizers for the KWC type energy as stated in Corollary \ref{CSJ}.
By Proposition \ref{PCal}, for the domain length $L=1$, the critical fidelity parameter where a one-jump minimizer and a two-jump minimizer have exactly the same minimal energy is given by $\lambda = 32/(1 \cdot 2 \cdot 3) = 16/3 \approx 5.333$.

Setting $\lambda = 16/3$, we compute the gradient flow of the phase-field KWC model using two different initial conditions corresponding to $m=1$ and $m=2$. 
To clearly illustrate the basin of attraction and the self-organizing nature of the gradient flow, we do not start exactly from the theoretically perfect step functions. Instead, we employ ``blended'' initial states defined by $u(x,0) = 0.5g(x) + 0.5u_m(x)$ (where $u_m$ is the theoretical minimizer defined in \eqref{EStm} for $m=1$ and $m=2$). These blended states are imperfect, possessing smaller jump sizes and residual slopes.

As shown in Figure \ref{Fig:Exp2}, starting from these imperfect blended states, the gradient flows successfully evolve, overcome the residual slopes, and perfectly reconstruct the two different theoretically predicted global minimizers.
This result visually and numerically supports the non-uniqueness caused by the non-convexity of the energy, verifying that multiple distinct global minimizers stably exist for the same critical parameter $\lambda$.

\begin{figure}[tb]
\centering 
\includegraphics[width=\linewidth]{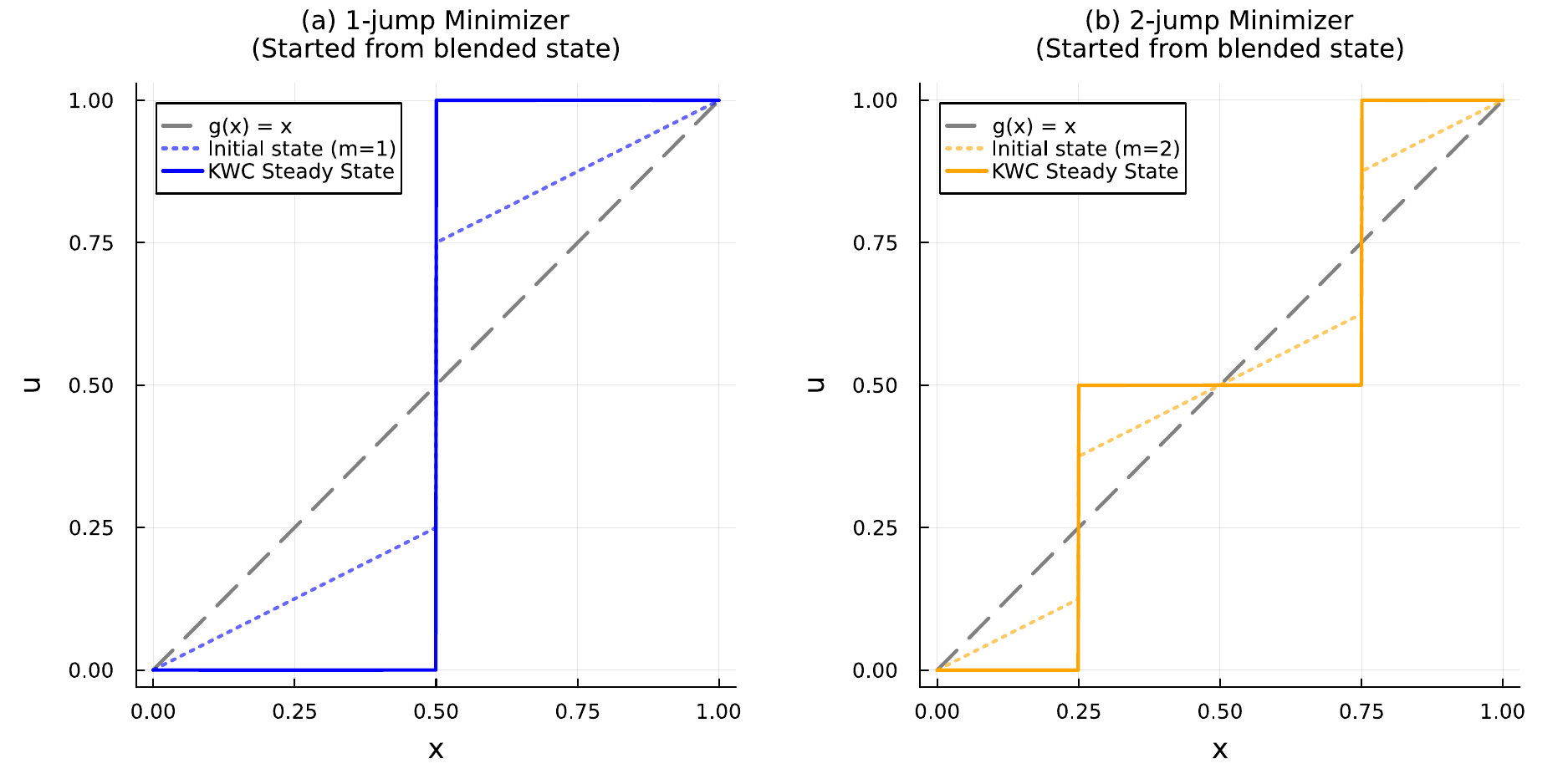}
\caption{Non-uniqueness of minimizers in the phase-field KWC model for $g(x)=x$ at the theoretical critical parameter $\lambda = 16/3$. By starting from blended imperfect states (dashed lines) with smaller jump sizes, the gradient flows successfully reconstruct two different global minimizers (solid lines): a one-jump state (a) and a two-jump state (b).}
\label{Fig:Exp2}
\end{figure}

\subsection{Segmentation of Non-monotone Oscillating Data}
To highlight the practical advantage of the KWC type energy for general signal segmentation tasks, we apply the models to a non-monotone oscillating signal, $g(x) = \sin(3\pi x)$.
For this experiment, we set the fidelity parameter to $\lambda = 150$ while keeping the other parameters unchanged. 
Unlike the theoretical verifications in the previous subsections where the Dirichlet boundary conditions were strictly imposed to match the specific assumptions of Theorem \ref{TSJG}, here we employ the homogeneous Neumann boundary condition for both $u$ and $v$. This is the standard setting in image and signal processing.

Figure \ref{Fig:Exp3} illustrates the comparison of the gradient flows starting from $u(x,0) = g(x)$.
The ROF model inherently suffers from the staircasing effect, generating numerous fine, irregular artificial jumps (false edges) to approximate the steep slopes of the sine wave. 
Interestingly, the AT model (Mumford-Shah) effectively captures the steepest gradients (the zero-crossings) by sharply dropping the phase-field variable $v \to 0$ and creating large jump discontinuities. This is a mathematically correct behavior to avoid massive Dirichlet energy penalties. However, because the Mumford-Shah energy favors piecewise smooth solutions, the regions between the jumps remain curved, failing to produce a piecewise constant clustering.

In striking contrast, the phase-field KWC model powerfully partitions the oscillating data into macroscopic, completely flat blocks. It dynamically cuts off the continuous peaks and valleys, clustering them into distinct flat plateaus separated by sharp phase-field boundaries.
This result visually and strongly demonstrates that the KWC model robustly achieves perfect piecewise constant segmentation even for highly oscillating non-monotone data, which is strictly consistent with the theoretical guarantee provided in Corollary \ref{CMain2}.

\begin{figure}[tb]
\centering 
\includegraphics[width=\linewidth]{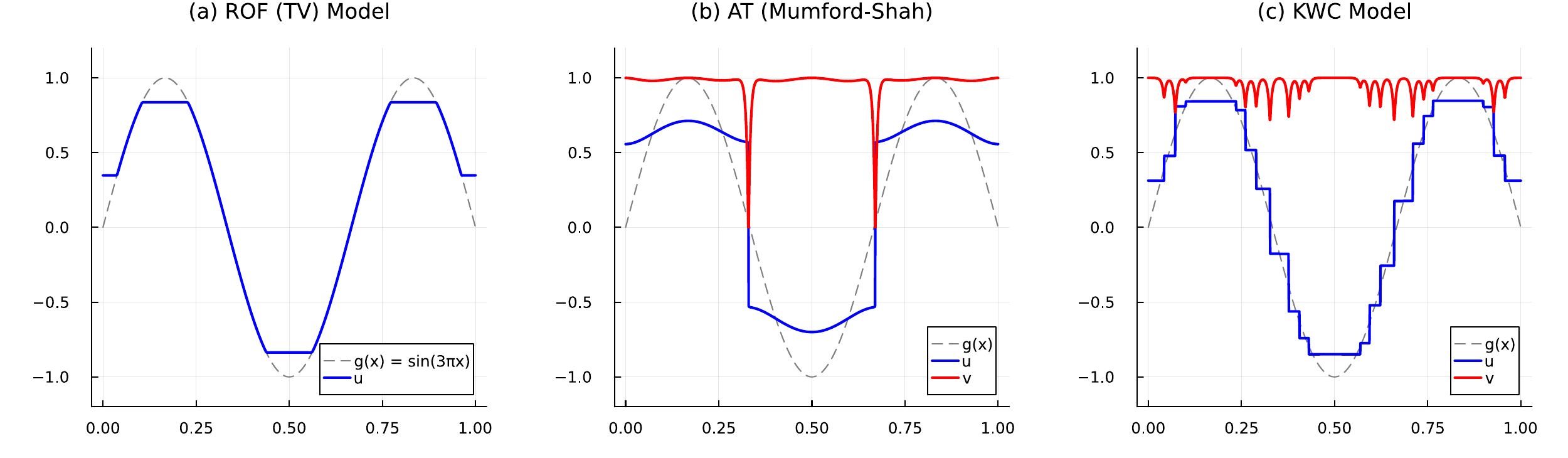}
\caption{Segmentation of non-monotone data $g(x) = \sin(3\pi x)$ under Neumann boundary conditions with $\lambda = 150$. (a) The ROF model exhibits severe staircasing. (b) The AT model creates jumps at the steepest points to save gradient energy but leaves the segments smoothly curved. (c) The KWC model robustly flattens the oscillating wave into perfectly piecewise constant blocks, demonstrating its unique and powerful clustering capability.}
\label{Fig:Exp3}
\end{figure}

\subsection{Robust Denoising of Piecewise Constant Signals}
As a final experiment, we demonstrate the robustness of the KWC model against severe noise in a typical denoising and segmentation task. 
We construct a ground-truth piecewise constant signal $u_{true}(x)$ consisting of three distinct flat plateaus. The observable input data $g(x)$ is then generated by corrupting $u_{true}(x)$ with heavy additive Gaussian noise. It should be emphasized that the gradient flows are computed solely using the noisy data $g(x)$ in the fidelity term, without any prior knowledge of the true signal.
For this experiment, we set the fidelity parameter to $\lambda = 50$ while keeping the other common parameters unchanged.

Figure \ref{Fig:Exp4} compares the recovery results of the three models.
The ROF model gets easily distracted by the local fluctuations of the noise. Because it penalizes the total variation but does not inherently prefer a small number of jumps, it generates spurious microscopic stairs (severe staircasing) to over-fit the noisy data. 
The AT model (Mumford-Shah) attempts to smooth out the noise, but as a consequence, it blurs the sharp edges of the underlying true signal, compromising the exact localization of the jumps.

However, the KWC model exhibits extraordinary robustness. Its non-convex jump energy essentially ignores the microscopic noisy fluctuations, forcing the segments to remain perfectly flat. Consequently, the KWC model flawlessly reconstructs the original piecewise constant ground truth, preserving both the perfectly flat plateaus and the perfectly sharp edges. 
This striking result numerically confirms the significant potential and practical advantage of KWC-type phase-field models for robust image and signal processing applications.

\begin{figure}[tb]
\centering 
\includegraphics[width=\linewidth]{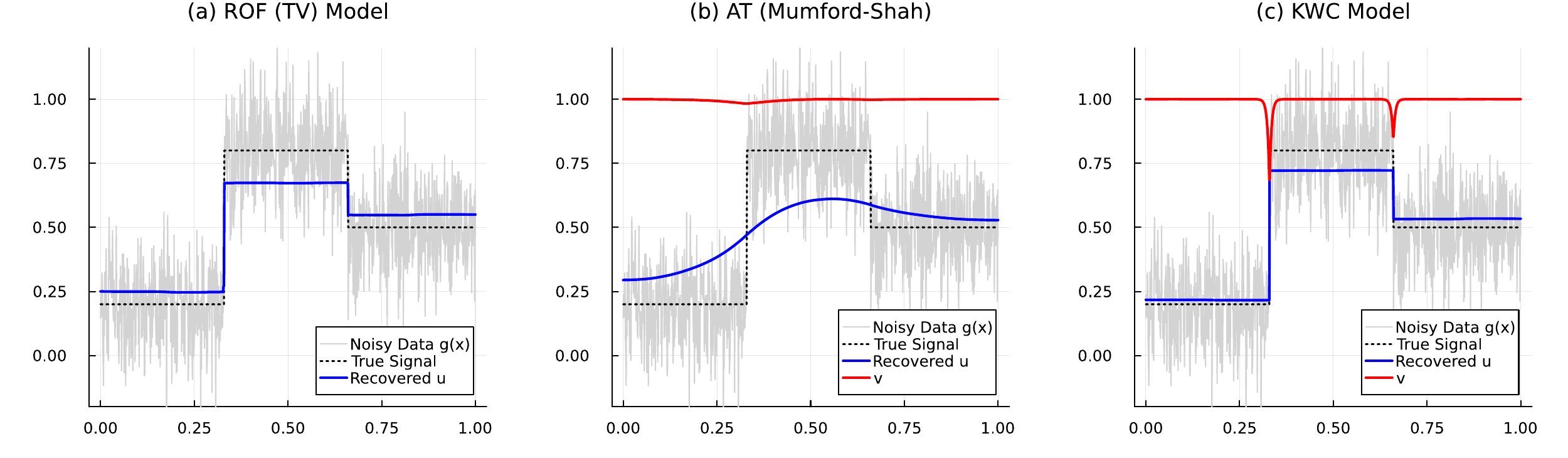}
\caption{Comparison of denoising performance on a noisy step function with $\lambda = 50$. The black dotted line represents the true signal, and the gray solid line is the noisy input data $g(x)$. (a) The ROF model suffers from staircasing due to the noise. (b) The AT model blurs the sharp edges. (c) The KWC model completely suppresses the noise and perfectly reconstructs the flat plateaus and sharp edges.}
\label{Fig:Exp4}
\end{figure}

\section*{Acknowledgements}
The authors are grateful to Professor Jun Okamoto for his valuable comments on Section~\ref{SA}.
 The work of the first author was partly supported by JSPS KAKENHI Grant Numbers JP24K00531 and JP24H00183 and by Arithmer Inc., Daikin Industries, Ltd.\ and Ebara Corporation through collaborative grants.
 The work of the fourth author was partly supported by JSPS KAKENHI Grant Numbers JP22K03425, JP22K18677, JP23H00086, JP25K00918.

\end{document}